\documentclass[a4paper,12pt]{article}
\PassOptionsToPackage{dvipsnames}{xcolor}

\usepackage[utf8]{inputenc}
\usepackage[english]{babel}

\usepackage{subcaption}
\usepackage{datetime}
\usepackage{enumerate}

\usepackage{amsfonts,amsmath,amssymb,amsbsy,amsthm}
\usepackage{mathtools}
\usepackage[T1]{fontenc}
\usepackage{lmodern}

\usepackage{floatrow}

\usepackage{verbatim}

\usepackage[backref=page]{hyperref}
\hypersetup{colorlinks,
	linkcolor={red!60!black},
	citecolor={green!60!black},
	urlcolor={blue!60!black}
}

\usepackage{tikz}
\usetikzlibrary{calc,positioning,decorations.pathmorphing,decorations.pathreplacing}
\tikzset{emp/.style={double distance = 0.3ex}}
\tikzset{oriented/.style={->,shorten >= 1.5pt}}

\tikzset{emp/.style={double distance = 0.3ex}}

\tikzset{M edge/.style={line width=1.3pt,double distance=1.1pt,->}}
\tikzset{F1 edge/.style={line width=1.3,color=red,->}}
\tikzset{F2 edge/.style={line width=1.3,color=blue,->}}
\tikzset{E edge/.style={line width=1.3,color=black,-}}

\tikzset{squared black vertex/.style={draw,minimum size=2mm,inner sep=0pt,outer sep=3pt,fill=black, color=black}}
\tikzset{red vertex/.style={circle,draw,minimum size=2mm,inner sep=0pt,outer sep=4pt,fill=red, color=red}}
\tikzset{blue vertex/.style={circle,draw,minimum size=2mm,inner sep=0pt,outer sep=4pt,fill=blue, color=blue}}
\tikzset{black vertex/.style={circle,draw,minimum size=2mm,inner sep=0pt,outer sep=3pt,fill=black, color=black}}
\tikzset{white vertex/.style={circle,draw,minimum size=2mm,inner sep=0pt,outer sep=3pt, color=black}}

\tikzstyle{edge}=[line width=1.3]

\tikzstyle{color1}=[color=red] 
\tikzstyle{color2}=[color=blue]
\tikzstyle{color3}=[color=green]

\tikzstyle{fucking loosely dotted}=[dash pattern=on \pgflinewidth off 6pt]
\tikzstyle{nonedge}=[color=red,fucking loosely dotted]

\tikzstyle{possible edge}=[edge, dashed]

\tikzstyle{snake}=[decorate, decoration=snake, segment length=1cm]
\tikzstyle{short snake}=[decorate, decoration=snake, segment length=7mm]
\tikzstyle{long snake}=[decorate, decoration=snake, segment length=11mm]

\colorlet{setfilling}{green!5!white}
\colorlet{setborder}{gray}

\usepackage{pgfplots}

\usepackage{lineno}
\newcommand*\patchAmsMathEnvironmentForLineno[1]{\expandafter\let\csname old#1\expandafter\endcsname\csname #1\endcsname
	\expandafter\let\csname oldend#1\expandafter\endcsname\csname end#1\endcsname
	\renewenvironment{#1}{\linenomath\csname old#1\endcsname}{\csname oldend#1\endcsname\endlinenomath}}\newcommand*\patchBothAmsMathEnvironmentsForLineno[1]{\patchAmsMathEnvironmentForLineno{#1}\patchAmsMathEnvironmentForLineno{#1*}}\AtBeginDocument{\patchBothAmsMathEnvironmentsForLineno{equation}\patchBothAmsMathEnvironmentsForLineno{align}\patchBothAmsMathEnvironmentsForLineno{flalign}\patchBothAmsMathEnvironmentsForLineno{alignat}\patchBothAmsMathEnvironmentsForLineno{gather}\patchBothAmsMathEnvironmentsForLineno{multline}}

\usepackage[abbrev,msc-links,backrefs]{amsrefs}
\usepackage{doi}

\renewcommand{\PrintDOI}[1]{\doi{#1}}

\newtheorem{theorem}{Theorem}
\newtheorem*{theorem-no-number}{Theorem}
\newtheorem{conjecture}[theorem]{Conjecture}
\newtheorem{lemma}[theorem]{Lemma}
\newtheorem{corollary}[theorem]{Corollary}
\newtheorem{claim}{Claim}
\newtheorem{subclaim}{Subclaim}[claim]
\newtheorem{remark}{Remark}

\theoremstyle{definition}
\newtheorem{definition}{Definition}

\newcommand{\pn}{{\rm pn}}

\newcommand{\D}{\mathcal{D}}

\newcommand{\floor}[1]{\lfloor #1 \rfloor}
\newcommand{\ceil}[1]{\lceil #1 \rceil}

\RequirePackage{amsmath}
\RequirePackage{amsfonts}

\renewcommand{\floor}[1]{\left\lfloor #1 \right\rfloor}

\renewcommand{\ceil}[1]{\left\lceil #1 \right\rceil}

\newcommand\tand{\ \text{and}\ }
\let\:\colon

\newcommand{\cD}{\mathcal{D}}

\newcommand{\cI}{\mathcal{I}}

\newcommand{\cS}{\mathcal{S}}
\newcommand{\cT}{\mathcal{T}}

\newcommand{\deleset}{\setminus}
\newcommand{\dele}{\setminus}
\newcommand{\addeset}{+}
\newcommand{\delvset}{-}
\newcommand{\delv}{-}

\title{Towards Gallai's path decomposition conjecture}

\author{F. Botler\textsuperscript{1} \hspace{.5cm}  M. Sambinelli\textsuperscript{2}\\
    {\footnotesize \textsuperscript{1}Programa de Engenharia de Sistemas e Computação}\vspace{-.2cm}\\
    {\footnotesize Universidade Federal do Rio de Janeiro}\\
    {\footnotesize \textsuperscript{2}Centro de Matemática, Computação e Cognição}\vspace{-.2cm}\\
    {\footnotesize Universidade Federal do ABC}
    \footnote{
{This study was financed in part by the Coordenação de Aperfeiçoamento de Pessoal de Nível Superior - Brasil (CAPES) - Finance Code 001.
F. Botler is partially supported by CNPq (Grant 423395/2018-1).
M. Sambinelli is supported by FAPESP (Grant 2017/23623-4) and CNPq (Grant 423833/2018-9).
E-mails: fbotler@cos.ufrj.br (F. Botler), m.sambinelli@ufabc.edu.br (M. Sambinelli).
}}}

\shortdate
\yyyymmdddate
\settimeformat{ampmtime}
\date{\today, \currenttime}
\date{}

\newcommand{\EV}[1]{EV(#1)}

\newcommand{\sD}{\mathcal{D}}
\let\simeq\undefined
\begin{document}

\sloppy

% \linenumbers

\maketitle

\newcommand{\calG}{the family of graphs for which (i) each block has maximum degree at most~\(3\); 
    and (ii) each component either has maximum degree at most \(3\) or has at most one block that contains triangles}

\begin{abstract}
    A path decomposition of a graph \(G\) is a collection of edge-disjoint paths of \(G\) that covers the edge set of \(G\).
    Gallai (1968) conjectured that every connected graph on~\(n\) vertices admits a path decomposition of cardinality at most \(\ceil{n/2}\).
    Seminal results towards its verification consider the graph obtained from \(G\) by removing its vertices of odd degree, which is called the \emph{E-subgraph} of \(G\).
    Lovász (1968) verified Gallai's Conjecture for graphs whose E-subgraphs consist of at most one vertex, and Pyber (1996) verified it for graphs whose E-subgraphs are forests.
In 2005, Fan verified Gallai's Conjecture for graphs in which each block of their E-subgraph is triangle-free and has maximum degree at most \(3\).
    Let \(\mathcal{G}\) be \calG.
In this paper, we generalize Fan's result by
    verifying  Gallai's Conjecture for graphs whose E-subgraphs are subgraphs of graphs in \(\mathcal{G}\).
This allows the components of the E-subgraphs to contain any number of blocks with triangles as long as they are subgraphs of graphs in \(\mathcal{G}\).

    \bigskip
    \noindent \textit{Keywords:} Graph, path, decomposition, Gallai's Conjecture, E-subgraph, maximum degree.
\end{abstract}

\section{Introduction}\label{sec:introduction}

In this paper, all graphs considered are finite and simple, i.e., contain a finite number of vertices and edges and have neither loops nor multiple edges.
A \emph{path decomposition} \(\D\) of a graph \(G\) is a collection of edge-disjoint paths of \(G\) that covers all the edges of \(G\).
A path decomposition~\(\D\) of a graph~\(G\) is \emph{minimum} if for every path decomposition~\(\D'\) of~\(G\) we have \(|\D| \leq |\D'|\), and the cardinality of such a minimum path decomposition, denoted by \(\pn(G)\), is called the \emph{path number} of \(G\).
Gallai proposed the following conjecture~(see~\cite{Lovasz68, Bondy14}).

\begin{conjecture}[Gallai, 1968]\label{conj:gallai}
If \(G\) is a connected simple graph on \(n\) vertices, then \(\pn(G) \leq \ceil{n/2}\).
\end{conjecture}

Lov\'asz~\cite{Lovasz68} verified Conjecture~\ref{conj:gallai} for graphs that have at most one vertex with even degree.
Pyber~\cite{Pyber96} extended Lov\'asz's result by proving that Conjecture~\ref{conj:gallai} holds for graphs in which each cycle contains at least one vertex with odd degree.
In 2005, Fan~\cite{Fan05} strengthen these results by extending Lovász's technique (see Section~\ref{sec:lemmas}).
Given a graph \(G\), the \emph{even subgraph} of \(G\) (\emph{E-subgraph}, for short), denoted by \(\EV{G}\), is the graph obtained from \(G\) by removing its vertices with odd degree, or, equivalently, the subgraph of \(G\) induced by its even degree vertices.
A \emph{block} in a graph \(G\) is a maximal \(2\)-connected subgraph of \(G\),
and a \emph{leaf block} of \(G\) is a block that contains at most one cut-vertex of \(G\).
Thus, the results above may be restated as follows.

\begin{theorem}[Lovász, 1968; Pyber, 1996; Fan, 2005]\label{thms:seminal}
    Let \(G\) be a connected graph on \(n\) vertices.    
    \begin{enumerate}[(a)]
        \item\label{thm:lovasz} 
        If \(\EV{G}\) contains at most one vertex, 
        then \(\pn(G)\leq\lfloor n/2\rfloor\);
        \item\label{thm:pyber} 
        If \(\EV{G}\) is a forest, 
        then \(\pn(G)\leq\lfloor n/2\rfloor\);
        \item\label{thm:fan} 
        If each block of \(\EV{G}\) is triangle-free 
        and has maximum degree at most \(3\),\\ 
        then \(\pn(G)\leq\lfloor n/2\rfloor\).
    \end{enumerate}
\end{theorem}

Given a vertex \(u\) with even degree in \(G\), the degree \(d_{\EV{G}}(u)\) is often referred as the \emph{E-degree} of \(u\).
Consider a graph \(G\) for which each block of \(\EV{G}\) has maximum degree at most \(3\) (and that may contain triangles),
and that has a minimum number of edges.
One can show that each leaf block of \(\EV{G}\) must be a triangle, this is a key idea in the proof of Theorem~\ref{thms:seminal}\eqref{thm:fan}.
In this paper we extend Theorem~\ref{thms:seminal}\eqref{thm:fan} by presenting a strategy for dealing with these remaining triangles in the following special case.
Let \(\mathcal{G}\) denote \calG.
Note that \(\mathcal{G}\) contains (properly) the E-subgraphs of the graphs considered in~\cite{Fan05}.
We verify Conjecture~\ref{conj:gallai} for graphs whose E-subgraphs are subgraphs of graphs in \(\mathcal{G}\).

\begin{theorem}\label{thm:theorems-weak}
	If \(G\) is a connected graph on \(n\) vertices such that \(\EV{G}\) is a subgraph of a graph in \(\mathcal{G}\),
	then \mbox{\(\pn(G)\leq\ceil{n/2}\)}.
\end{theorem}

We remark that Theorem~\ref{thm:theorems-weak} extends Theorem~\ref{thms:seminal}\eqref{thm:fan} by allowing some blocks  to contain triangles.
More specifically, a component of the E-subgraph may contain any number of blocks that contain triangles as long as it is a subgraph of graph  in \(\mathcal{G}\) (see Figure~\ref{fig:many-triangles}).

\begin{figure}
    \centering
    \begin{subfigure}{.45\textwidth}
        \centering
        \begin{tikzpicture}[scale = .5]

    \draw[color=white] (0, 5) rectangle (2, -3.5);

    \begin{scope}[scale=2, yshift=-1.1cm]
	\node (0)  [black vertex] 	at (0,0) {};
	\node (1)  [black vertex] 	at (1,0) {};
	\node (2)  [black vertex] 	at (2,0) {};
	\node (3)  [black vertex] 	at (3,0) {};
	\node (4)  [black vertex] 	at (0,1) {};
	\node (5)  [black vertex] 	at (3,1) {};
	\node (6)  [black vertex] 	at (0,2) {};
	\node (7)  [black vertex] 	at (3,2) {};
	\node (8)  [black vertex] 	at (0,3) {};
	\node (9)  [black vertex] 	at (1,3) {};
	\node (10) [black vertex] 	at (2,3) {};
	\node (11) [black vertex] 	at (3,3) {};
	    
    \end{scope}

    \draw[edge, black] (0) -- (1) -- (4) -- (0);
    \draw[edge, black] (2) -- (3) -- (5) -- (2);
    \draw[edge, black] (6) -- (8) -- (9) -- (6);
    \draw[edge, black] (7) -- (10) -- (11) -- (7);

    \draw[edge, black] (1) -- (2);
    \draw[edge, black] (4) -- (6);
    \draw[edge, black] (9) -- (10);
    \draw[edge, black, dotted] (5) -- (7);
    
    \draw[edge,black] (0) -- ($(0)+(180:1)$) (0) -- ($(0)+(225:1)$) (0) -- ($(0)+(270:1)$);
    \draw[edge,black] (3) -- ($(3)+(0:1)$) (3) -- ($(3)+(-45:1)$) (3) -- ($(3)+(-90:1)$);
    \draw[edge,black] (8) -- ($(8)+(180:1)$) (8) -- ($(8)+(135:1)$) (8) -- ($(8)+(90:1)$);
    \draw[edge,black] (11) -- ($(11)+(0:1)$) (11) -- ($(11)+(45:1)$) (11) -- ($(11)+(90:1)$);
    
\end{tikzpicture}

 \end{subfigure}
    \begin{subfigure}{.45\textwidth}
        \centering
        \begin{tikzpicture}[scale = 0.5]

    \draw[color=white] (0, 5) rectangle (2, -3.5);

    \begin{scope}
        \node (x0) [] 	at (0,0) {};
        \node (x1) [black vertex] 	at ($(x0)+(90:1)$) {};
        \node (x2) [black vertex] 	at ($(x0)+(210:1)$) {};
        \node (x3) [black vertex] 	at ($(x0)+(330:1)$) {};

        \draw[edge, black] (x1) -- (x2) -- (x3) -- (x1);
    \end{scope}
	
    \begin{scope}[yshift=4cm, rotate=180]
        \node (y0) [] 	at (0,0) {};
        \node (y1) [black vertex] 	at ($(y0)+(90:1)$) {};
        \node (y2) [black vertex] 	at ($(y0)+(210:1)$) {};
        \node (y3) [black vertex] 	at ($(y0)+(330:1)$) {};

        \draw[edge, black] (y1) -- (y2) -- (y3) -- (y1);
        \draw[edge, black] (y3) -- ($(y3)+(-45:1)$) (y3) -- ($(y3)+(-90:1)$) (y3) -- ($(y3)+(-135:1)$);
    \end{scope}
	
    \begin{scope}[shift=(-30:4cm), rotate=60]
        \node (z0) [] 	at (0,0) {};
        \node (z1) [black vertex] 	at ($(z0)+(90:1)$) {};
        \node (z2) [black vertex] 	at ($(z0)+(210:1)$) {};
        \node (z3) [black vertex] 	at ($(z0)+(330:1)$) {};

        \draw[edge, black] (z1) -- (z2) -- (z3) -- (z1);
        \draw[edge, black] (z3) -- ($(z3)+(-45:1)$) (z3) -- ($(z3)+(-90:1)$) (z3) -- ($(z3)+(-135:1)$);
    \end{scope}
	
    \begin{scope}[shift=(-150:4cm), rotate=-60]
        \node (w0) [] 	at (0,0) {};
        \node (w1) [black vertex] 	at ($(w0)+(90:1)$) {};
        \node (w2) [black vertex] 	at ($(w0)+(210:1)$) {};
        \node (w3) [black vertex] 	at ($(w0)+(330:1)$) {};

        \draw[edge, black] (w1) -- (w2) -- (w3) -- (w1);
        \draw[edge, black] (w3) -- ($(w3)+(-45:1)$) (w3) -- ($(w3)+(-90:1)$) (w3) -- ($(w3)+(-135:1)$);
    \end{scope}

    \draw[edge, black] (x1) -- (y1);
    \draw[edge, black] (x2) -- (w1);
    \draw[edge, black] (x3) -- (z1);

    \draw[edge, black, dotted] (y2) -- (z3);
    \draw[edge, black, dotted] (w2) -- (y3);
    \draw[edge, black, dotted] (z2) -- (w3);

\end{tikzpicture}

 \end{subfigure}
    \caption{Examples of graphs in which more than one block contain triangles, and that may be completed (by using the dotted lines) to graphs in \(\mathcal{G}\).}
    \label{fig:many-triangles}
\end{figure}
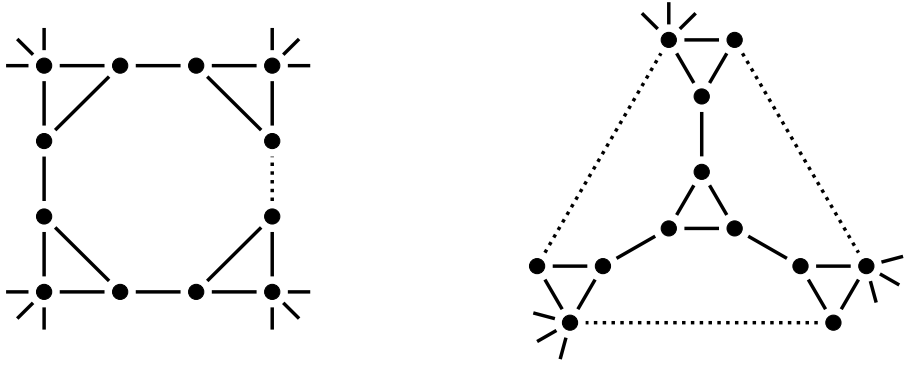

Conjecture~\ref{conj:gallai} has being deeply explored, and the literature indicating its correctness include results for Eulerian graphs with maximum degree at most \(4\)~\cite{FavaronKouider88}; a family of regular graphs~\cite{BotlerJimenez2017}; a family of triangle-free graphs~\cite{JimenezWakabayashi2017}; and maximal outerplanar graphs and \(2\)-connected outerplanar graphs~\cite{GengFangLi15}.
Recent results were obtained by Bonamy and Perrett~\cite{BonamyPerret19} who verified Conjecture~\ref{conj:gallai} for graphs with maximum degree at most \(5\).

Note that the results in Theorem~\ref{thms:seminal} give a bound of \(\lfloor n/2\rfloor\) for the graphs studied, which is slightly different from the bound of \(\lceil n/2\rceil\) proposed by Gallai. 
A straightforward condition for a graph on \(n\) vertices not to admit the former bound is to have sufficiently many edges.
More precisely, if \(|E(G)| > \lfloor n/2\rfloor (n-1)\), then we have \(\pn(G)\geq\lceil n/2\rceil\).
In this case, \(n\) must be an odd integer.
Such graphs are known as \emph{odd semi-cliques}~\cite{BonamyPerret19}.
This motivates the following strengthening of Conjecture~\ref{conj:gallai}, which was considered in~\cite{BoSaCoLe,BoJiSa18}.

\begin{conjecture}\label{conj:strong-gallai}
If \(G\) is a connected simple graph on \(n\) vertices, then either \(\pn(G) \leq \floor{n/2}\), or  \(\pn(G) = \ceil{n / 2}\) and \(G\) is an odd semi-clique.
\end{conjecture}

Graphs \(G\) for which \(\pn(G)\leq\floor{n/2}\) are called \emph{Gallai graphs}.
Botler, Coelho, Lee, and Sambinelli~\cite{BoSaCoLe} verified Conjecture~\ref{conj:strong-gallai} for graphs with treewidth at most~\(3\) by proving that a partial \(3\)-tree are either Gallai graphs, or one of the two odd semi-cliques that are partial \(3\)-trees (\(K_3\) and \(K_5-e\)).
They also prove~\cite{BoSaCoLe-arxiv} that a graph with maximum degree at most \(4\) is either a Gallai graph, or one of the three odd semi-cliques with maximum degree at most  \(4\) (\(K_3\), \(K_5-e\), and \(K_5\)).
More recently, Botler, Jiménez, and Sambinelli~\cite{BoJiSa18} verified Conjecture~\ref{conj:strong-gallai} for triangle-free planar graphs by proving that every such  graph is a Gallai graph.
In this paper, we explore an intermediate statement between Conjectures~\ref{conj:gallai} and~\ref{conj:strong-gallai}.
We prove that, for the classes of graphs studied, all graphs are Gallai graphs except for a special family \(\cS\).
The family \(\cS\), which we call the \emph{SET graphs} (see Section~\ref{sec:main-result}), and for which we check Conjecture~\ref{conj:gallai}, differs from previous families of exceptions for two reasons.
First, \(\cS\) contains an infinite number of odd semi-cliques.
Second, apart from containing odd semi-cliques, \(\cS\) also contains non odd semi-cliques which we cannot guarantee the bound of \(\floor{n/2}\).
We remark that this is the first result regarding E-subgraphs to handle odd semi-cliques.
Our result can be more specifically stated as follows.

\begin{theorem}\label{thm:theorems-strong}
	If \(G\) is a connected graph such that \(\EV{G}\) is a subgraph of a graph in \(\mathcal{G}\),
	then \(G\) is a Gallai graph or \(G\in\mathcal{S}\).
\end{theorem}

This work is organized as follows.
In Section~\ref{sec:lemmas} we present some technical lemmas.
In Section~\ref{sec:main-result}, we verify Conjecture~\ref{conj:gallai} for the special case of graphs whose E-subgraphs have maximum degree at most~\(3\);
In Section~\ref{sec:further}, we verify Conjecture~\ref{conj:gallai} for graphs \(G\) for which \(\EV{G}\) is a subgraph of a graph in \(\mathcal{G}\); and in Section~\ref{sec:concluding}, we present some concluding remarks.

\begin{comment}
\smallskip
\noindent\textbf{Overview of the proof.}\label{sec:overview}
{\color{red}
The proof of our main theorem (Theorem~\ref{thm:main-theorem}) consists of seven claims regarding a minimal counterexample for Conjecture~\ref{conj:gallai}.
Claim~\ref{claim:no-addible-sequence} says that such a graph cannot contain a sequence of addible edges (see Definition~\ref{def:addible-sequence}), while Claim~\ref{claim:no-unique-even-neighbor} is a consequence of Lemma~\ref{lemma:fan2} that says that no vertex contains precisely one neighbor with even degree.
Claim~\ref{claim:only-triangles} reveals the structure of the E-subgraph of a minimal counterexample to be a set of vertex-disjoint triangles and isolated vertices.
Claim~\ref{claim:one-triangle-neighbor} is the key claim of this proof, and consists of applying Fan's strategy on a vertex with odd degree, to find out that the neighbors with even degree of such a vertex are contained in the same component of the E-subgraph.
This implies that the E-subgraph contains no isolated vertex.
Claims~\ref{claim:full-vertices1} and~\ref{claim:full-vertices2} state that every vertex with odd degree must have a neighbor with even degree; and Claim~\ref{claim:full-vertices3} states that some of the vertices of odd degree must be connected to every vertex of a component of the E-subgraph.
Finally, we use this structure to show that this minimal counterexample satisfies Conjecture~\ref{conj:gallai}, obtaining a contradiction.}
\end{comment}

\smallskip
\noindent\textbf{Notation.}\label{sec:notation}
The basic terminology and notation used in this paper are standard (see, e.g.~\cite{BoMu08}). 
Given a graph \(G\), we denote its vertex set by \(V(G)\) and its edge set by \(E(G)\).
The set of neighbors of a vertex \(u\) in a graph \(G\) is denoted by \(N_G(u)\) and its degree by \(d_G(u)\).
When \(G\) is clear from the context, we simply write \(N(u)\) and  \(d(u)\).
Since \(G\) is simple, we always have \(d_G(u) = |N_G(u)|\).

A graph \(H\) is a \emph{subgraph} of a graph \(G\), denoted by \(H \subseteq G\), if \(V(H) \subseteq V(G)\) and \(E(H) \subseteq E(G)\).
Given a set of vertices \(X \subseteq V(G)\), we say that \(H\) is the subgraph of \(G\) \emph{induced by \(X\)}, denoted by \(G[X]\), if  \(V(H) = X\) and \(E(H) = \{xy \in E(G) \colon x,y \in X\}\).
Given a set of edges \(Y \subseteq E(G)\), we say that \(H\) is the subgraph of \(G\) \emph{induced by \(Y\)}, denoted by \(G[Y]\), if  \(E(H) = Y\) and \(V(H) = \{x \in V(G) \colon xy \in Y\}\).
For ease of notation, when convenient, we write simply \(Y\) to refer to the graph \(G[Y]\).
Given \(X \subseteq V(G)\), we define \(G \delvset X=G[V(G) \setminus X]\).
In the case that \(X = \{u\}\), we simply write \(G \delv u\).
Given a set~\(Y\) of edges, we define the graphs \(G \deleset Y=(V(G),E(G)\setminus Y)\).
As before, in the case that \(Y = \{e\}\), we simply write \(G \dele e\).

A \emph{path} \(P\) in a graph \(G\) is a sequence \(u_0u_1\cdots u_\ell\) of distinct vertices in \(V(G)\) such that~\(u_iu_{i+1}\in E(G)\), for~\(i=0,\ldots,\ell-1\).
We say that \(u_0\) and \(u_\ell\) are the \emph{end vertices} of \(P\), and that \(P\) \emph{joins} \(u_0\) and \(u_\ell\).
When convenient, we consider a path as the subgraph of~\(G\) induced by the set of edges \(\{u_iu_{i+1} \colon i = 0, \ldots, \ell - 1\}\).
A \emph{shortest} path joining two vertices \(u\) and \(v\) 
is a path that joins \(u\) and \(v\) with a minimum number of edges.

Given a vertex \(u\) of a graph \(G\), we say that \(u\) is an odd (resp. even) vertex if its degree is odd (resp. even).
Analogously, we say that a neighbor \(v\) of \(u\) is an odd (resp. even) neighbor of \(u\) if \(v\) has odd (resp. even) degree.
Given a path decomposition \(\D\) of a graph \(G\) and a vertex \(u\in V(G)\), we denote by \(\D(u)\) the number of paths in \(\D\) that have \(u\) as an end vertex.
It is not hard to check that \(\D(u) \equiv d(u) \pmod{2}\).
In particular, if \(u\) is an odd vertex, we have \(\D(u) \geq 1\), for any path decomposition \(\D\) of \(G\).

\section{Technical Lemmas}\label{sec:lemmas}

In this section we present some technical results used throughout our proof.
Following the strategy presented by Fan~\cite{Fan05}, our technique relies in the following definition.

\begin{definition}\label{def:fan-addible}
    Let \(u\) be a vertex in a graph \(G\) and let \(B\) be a set of edges incident to~\(u\). 
    Let \(G' = G \setminus B\),  let \(\D'\) be a path decomposition of \(G'\),
    and let \(A = \{ux_i \colon 1 \leq i \leq k\}\) be a subset of \(B\).
    We say that \(A\) is \emph{addible towards \(u\)} (resp.\ \emph{addible outwards \(u\)}) \emph{with respect to \(\D'\)} if \(G' + A\) admits a path decomposition \(\D\) such that
    \begin{enumerate}[(i)]
        \item\label{def:fan-addible1} \(|\D|=|\D'|\);
        \item\label{def:fan-addible2} \(\D(u) = \D'(u) + |A|\) and \(\D(x_i) = \D'(x_i) -1\), for \(1\leq i\leq k\) 
        \item[]                       (resp.\ \(\D(u) = \D'(u) - |A|\) and \(\D(x_i) = \D'(x_i) +1\), for \(1\leq i\leq k\));
        \item\label{def:fan-addible3} \(\D(v) = \D'(v)\) for each \(v \in V(G)\setminus \{u,x_1,\ldots,x_k\}\).
    \end{enumerate}
    In these cases, we say that  \(\D\) is an \emph{\(A\)-transformation} of \(\D'\) \emph{towards} (resp.\ \emph{outwards}) \(u\).
    For simplifying the notation, when \(k = 1\), we write \(ux_1\)-transformation instead of \(\{ux_1\}\)-transformation.
\end{definition}

Definition~\ref{def:fan-addible} is an extension of the definition of addible~\cite[Definition~3.2]{Fan05}.
In fact, the present definition of \emph{addible towards} precisely matches the definition of \emph{addible at} given by Fan~\cite{Fan05}.
The next observation is used frequently in our proof.
 \begin{remark}\label{remark:addible-concatenation}
Let \(B\) be a set of edges incident to a vertex \(u\) of a graph \(G\).
     Let \(\D'\) be a path decomposition of \(G'=G\setminus B\), and
     \(A \subset B\) be an addible set towards (resp.\ outwards) \(u\) with respect to \(\D'\),
     and  let \(\D''\) the an \(A\)-transformation of \(\D'\).
     If \(A'  \subset B \setminus A\) is addible towards (resp.\ outwards) \(u\) with respect to \(\D''\),
     then \(A\cup A'\) is addible towards (resp.\ outwards) \(u\) with respect to \(\D'\).
 \end{remark}

The next two lemmas are results of Fan~\cite[Lemmas 3.4 and 3.6]{Fan05}. 

\begin{lemma}\label{lemma:fan2}
    Let \(G\) be a graph and \(uv \in E(G)\). 
    Suppose that \(\D'\) is a path decomposition of \(G' = G \setminus {uv}\). 
    If \(\D'(v) > |\{w\in N_{G'}(u) \colon \D'(w) = 0\}|\), then \(uv\) is addible towards \(u\) with respect to \(\D'\).
\end{lemma}

Lemma~\ref{lemma:fan2} motivates the following definitions.
Let \(\D\) be a path decomposition of a graph \(G\) and let \(u\) be a vertex of \(G\).
We say that a vertex \(u\) is \emph{passing} in \(\D\) if \(\D(u) = 0\).
Thus, we say that vertices in \(\{v\in N_{G'}(u) \colon \D'(v) = 0\}\) are the \emph{passing neighbors} of \(u\) in \(\D'\).
Lemma~\ref{lemma:fan2} then says that
if in a path decomposition \(\D'\) of \(G\dele uv\) there are more paths having \(v\) 
as end vertex than passing neighbors of \(u\), 
then \(uv\) is addible towards \(u\) with respect to \(\D'\).

\begin{lemma}\label{lemma:fan4}
    Let \(u\) be a vertex in a graph \(G\) and let \(G' = G\setminus\{ux_1,\ldots,ux_h\}\), where \(x_i\in N_G(u)\).
    Suppose that \(\D'\) is a path decomposition of \(G'\) with \(\D'(v)\geq 1\) for every \(v\in N_G(u)\). 
    Then, for any \(x\in\{x_1,\ldots,x_h\}\), there is \(A\subseteq\{ux_1,\ldots,ux_h\}\) such that
    \begin{itemize}
        \item \(ux\in A\) and \(|A|\geq\left\lceil \frac{h}{2}\right\rceil\); and
        \item \(A\) is addible towards \(u\) with respect to \(\D'\)
    \end{itemize}
\end{lemma}

Given a graph \(G\) and a set \(M\) of edges of \(G\), we denote by \(G[M]\) the subgraph of \(G\) induced by the vertices incident to edges of \(M\).
We say that a set of edges \(M\subseteq E(G)\) is an \emph{induced matching} if the set edges of \(G[M]\) is precisely \(M\).
Note that every subset of an induced matching is also an induced matching.
The next two lemmas are used in the proof of Claim~\ref{claim:one-triangle-neighbor}.

\begin{lemma}\label{lemma:induced-matching}
    Let \(G\) be a graph and let \(M=\{e_1,\ldots, e_k\}\) be an induced matching in \(G\), where \(e_i = u_iv_i\) for \(i = 1, \ldots, k\).
    Let \(\D'\) be a path decomposition of \(G' = G \setminus M\).
    If, for \(i=1,\ldots, k\), the vertex \(u_i\) has no passing neighbor in \(\D'\) and \(\D'(v_i)\geq 1\),
    then there is a path decomposition \(\D\) of \(G\)
    such that 
    \begin{itemize}
        \item \(|\D| = |\D'|\);
        \item  for \(i=1,\ldots, k\), we have \(\D(u_i) = \D'(u_i)+1\) and \(\D(v_i) = \D'(v_i)-1\); and
        \item \(\D(w) = \D'(w)\) for every \(w\in V(G)\setminus V(M)\).
    \end{itemize}
\end{lemma}
\begin{proof}
    The proof follows by induction on the size of \(M\).
    If \(k=0\), then \(M=\emptyset\), and the statement holds with \(\D = \D'\).
    Thus, assume \(k\geq 1\).
    Since \(u_1\) has no passing neighbor in \(\D'\),
    and \(\D'(v_1)\geq 1\),
    by Lemma~\ref{lemma:fan2}, \(e_1\) is addible towards \(u_1\) with respect to \(\D'\).
    Let \(G''=G'+e_1\), and let \(\D''\) be the \(e_1\)-transformation of \(\D'\) towards \(u_1\).
    Note that \(\D''(u_1) = \D'(u_1)+1\), \(\D''(v_1) = \D'(v_1)-1\), and \(\D''(w) = \D'(w)\), for every \(w \in V(G) \setminus \{u_1, v_1\}\).
    Thus, \(v_1\) is the only vertex that possibly became a passing vertex through the transformation from \(\D'\) to \(\D''\).
    Now, consider \(M' = M \setminus \{e_1\} = \{e_2,\ldots,e_k\}\).
    Clearly, \(M'\) is an induced matching and \(G''  = G \deleset M'\).
    Since \(\{e_1,\ldots, e_k\}\) is an induced matching,
    \(v_1\) is not adjacent to \(u_2,\ldots,u_k\) in \(G''\).
    Thus \(u_i\) has no passing neighbor in \(\D''\) and \(\D''(v_i)\geq 1\), for \(i=2,\ldots,k\),
    and hence, by induction hypothesis,  there is a path decomposition \(\D\) of \(G\)
    such that \(|\D| = |\D''| = |\D'|\);
    for \(i=2,\ldots,k\), we have \(\D(u_i) = \D''(u_i)+1 = \D'(u_i)+1\) and \(\D(v_i) = \D''(v_i)-1 = \D'(v_i)-1\);
    and \(\D(w) = \D''(w)\) for every \(w\in V(G)\setminus\{u_i,v_i\colon i=2,\ldots,k\}\).
    Note that \(\D(u_1) = \D''(u_1)= \D'(u_1)+1\) and \(\D(v_1) = \D''(v_1) = \D'(v_1)-1\),
    and \(\D(w) = \D''(w)=\D'(w)\) for every \(w\notin\{u_i,v_i\colon i=1,\ldots,k\}\),
    as desired.
\end{proof}

It is not hard to check that the condition that \(\{e_1,\ldots,e_k\}\) is an induced matching in  Lemma~\ref{lemma:induced-matching}
may be easily replaced by the condition that 
\(N(u_i)\cap\{v_j\colon j=1,\ldots,k\} = \{v_j\}\),
i.e., 
\(u_i\)'s may be adjacent to other \(u_i\)'s
and \(v_i\)'s may be adjacent to other \(v_i\)'s,
but each \(u_i\) is adjacent to precisely one \(v_j\), namely, \(v_i\).

In what follows, given a graph \(G\) and a vertex \(u\),
we denote by \(E_G(u)\) the set of edges of \(G\) that are incident to \(u\).
The next lemma says, roughly, that after removing a set \(E'\subseteq E_G(u)\) of edges incident to an odd vertex \(u\),
and applying Lemma~\ref{lemma:fan4},
any further transformation outward \(u\) leaves more paths having \(u\) as an end vertex,
than the number of edges that remain to add.

\begin{lemma}\label{lemma:extra-edges}
    Let \(u\) be an odd vertex in a graph \(G\), \(B \subseteq E_G(u)\), and \(\D_0\) be a path decomposition of \(G \setminus B\).
    Let \(A_1\subseteq B\) be an addible set towards \(u\) with respect to \(\D_0\),
    and let \(A_2\subseteq B\setminus A_1\) be an addible set outwards \(u\) with respect to 
    the \(A_1\)-transformation \(\D_1\) of \(\D_0\).
    If \(\D_2\) is the \(A_2\)-transformation of \(\D_1\) and \(|A_1| \geq \ceil{|B|/2}\),
    then \(\D_2(u) > |B|-|A_1\cup A_2|\).
\end{lemma}
\begin{proof}
We claim that \(\D_1(u) \geq 1 + \lfloor |B|/2\rfloor\).
    Indeed, since \(d_G(u)\) is odd, if \(|B|\) is even, then \(\D_0(u)\geq 1\), and \(\D_1(u) \geq \D_0(u) + |B|/2 \geq 1 + |B|/2 = 1+\lfloor |B|/2\rfloor\); 
    if \(|B|\) is odd, then \(\D_1(u) \geq \lceil |B|/2\rceil = 1+\lfloor |B|/2\rfloor\).
Note that \(\D_2(u) = \D_1(u) - |A_2|\).
    Since \(|B| = \lfloor |B|/2\rfloor + \lceil |B|/2\rceil\), we have 
\begin{align*}
        \D_2(u) &= \D_1(u) - |A_2| \\ 
        &\geq 1 + \lfloor |B|/2\rfloor - |A_2| \\
        &= 1 + |B| - \lceil |B|/2\rceil - |A_2|\\
        &\geq 1 + |B| - |A_1| - |A_2|.\end{align*}\end{proof}

\subsection{Absorbing lemmas}

In this section, we present some lemmas that allow Gallai graphs to absorb special subgraphs while keeping its Gallai property.
The next remark is used often in our proofs.

\begin{remark}\label{remark:gallai-components}
	If every component of a graph \(G\) is a Gallai graph,
	then \(G\) is a Gallai graph.
\end{remark}

\begin{proof}
	Let \(H_1,\ldots,H_k\) be the components of \(G\).
Thus
	\(\pn(G) = \sum_{i=1}^k\pn(H_i) \leq \sum_{i=1}^k\floor{|V(H_i)|/2} 
	\leq \floor{|V(G)|/2}\).
	Therefore \(G\) is a Gallai graph.
\end{proof}

We say that a graph \(G\) is a \emph{single even triangle graph (SET graph)}
if \(\EV{G}\) is a triangle and every odd vertex of \(G\) has at least two
even neighbors.
Since a SET graph has three even vertices, it must have odd order.
Single even triangle graphs are special cases of our proof.
Note that the graphs obtained from a complete graph on \(2k+1\) vertices
by removing a matching of size \(k-1\) 
is a SET graph that is an odd semi-clique.
Although we were not able to fully characterize which SET graphs are Gallai graphs,
checking the validity of Conjecture~\ref{conj:gallai} for them 
is a straightforward task (see Lemma~\ref{lemma:ESET-special-decomposition}).
In what follows, we extend the definition of SET graphs,
and indicate some of their vertices to be \emph{connection vertices}.
We say that a graph \(K\) 
is an \emph{extended single even triangle graph (ESET graph)}
if one of the following hold.
\begin{enumerate}[i)]
	\item\label{def:eset-1}	
	    \(K\) is a SET graph.
		In this case, every vertex of \(K\) is a \emph{connection} vertex;
\item\label{def:eset-3}	
	    \(K\) is obtained from a SET graph \(K^-\)
		by adding a new vertex \(z\) adjacent to an odd and an even vertex
		of \(K^-\).
		In this case,  \(z\) is the \emph{connection} vertex of \(K\).
\end{enumerate}
We say that an ESET is of \emph{type} \ref{def:eset-1} or \ref{def:eset-3}, 
according to the items above.
Throughout the proof we are required to absorb ESET graphs without increasing the size of the path decomposition.
The next lemma provide, for ESET graphs, path decompositions that contain two paths 
that have a fixed connection vertex as end vertex.

\begin{lemma}\label{lemma:ESET-special-decomposition}
Let \(K\) be an ESET graph on \(n\) vertices.
    If \(u\) is a connection vertex of~\(K\), then \(K\) admits a path decomposition \(\D\) of \(K\) such that \(\D(u) \geq 2\) and \(|\D|\leq \lceil n/2\rceil\).
\end{lemma}

\begin{proof}
	Let \(K\), \(n\), and \(u\) be as in the statement.
	We divide the proof according to the type of \(K\).
	
	\smallskip
	\noindent{\bf Type~\ref{def:eset-1}.}
	By the definition of ESET graph of type~\ref{def:eset-1}, \(K\) is a SET graph and \(u\) is any of its vertices.
	Let \(S\) be the set of edges that join \(u\) 
	to the even vertices of \(K\),
	and let \(K' = K \setminus S\).
	Note that \(K'\) has precisely one even vertex,
	and hence, by Theorem~\ref{thms:seminal}\eqref{thm:lovasz},
	it follows that \(\pn(K')\leq\lfloor n/2\rfloor\).
	Let \(\D'\) be a minimum path decomposition of \(K'\).
	Note that every neighbor of \(u\) in \(K\) has odd degree in \(K'\).
	Thus, by Lemma~\ref{lemma:fan4}, there is \(B\subseteq S\), with \(|B| \geq |S|/2\),
	which is addible with respect to \(\D'\).
	Let \(\D''\) be the \(B\)-transformation of \(\D'\) towards \(u\).
	{\color{black}Note that \(S\setminus B\) contains at most one edge.
	If \(S\setminus B = \{e\}\), then let \(P\) 
	be the path containing only \(e\),
	and put \(\D = \D''\cup\{P\}\),
	otherwise, put \(\D = \D''\).}
	It is not hard to check that \(\D(u)\geq 2\) and \(|\D|\leq \lceil n/2\rceil\), as desired.

	\smallskip
	\noindent{\bf Type~\ref{def:eset-3}.}
	By the definition of ESET graph of type~\ref{def:eset-3}, 
	\(K\) is obtained from a SET graph~\(K^-\) by adding a new vertex \(u\) adjacent to 
    an odd vertex \(x\) and an even vertex \(y\) of \(K^-\).
    Since \(K^-\) is an ESET of type~\ref{def:eset-1},
    \(K^-\) admits a path decomposition \(\D^-\) such that \(\D^-(y)\geq 2\) and 
    \(|\D^-|\leq\ceil{(n-1)/2}\).
    Since \(x\) is an odd vertex of \(K^-\), we have \(\D^-(x)\geq 1\).
    Let \(P_x^-\) be a path in \(\D^-\) that has \(x\) as end vertex.
    Since \(\D^-(y)\geq 2\), there are two paths in \(\D^-\)
    that have \(y\) as end vertex,
    and hence there is at least one such path, say \(P_y^-\), that is different from \(P_x^-\).
    Let \(P_x = P_x^-+ux\) and \(P_y = P_y^-+yu\),
    and hence \(\big(\D^-\setminus\{P_x^-,P_y^-\}\big)\cup\{P_x,P_y\}\) is the desired decomposition.
\end{proof}

As a direct application of Lemma~\ref{lemma:ESET-special-decomposition}, we have the following result.

\begin{lemma}\label{lemma:ESET-absorbing}
	Let \(G\) be a graph that can be decomposed into 
	an ESET graph \(K\) on \(n\) vertices
	with a connection vertex \(u\in V(K)\),
	and a path \(P\) such that \(V(P)\cap V(K)=\{u\}\).
	Then \(\pn(G)\leq \lceil n/2\rceil\).
\end{lemma}

\begin{proof}
	Let \(G\), \(K\), \(P\), \(n\), and \(u\) be as in the statement.
	By Lemma~\ref{lemma:ESET-special-decomposition},
	\(K\) admits a path decomposition \(\D_K\) 
	such that \(\D_K(u)\geq 2\) and \(|\D_K|\leq \lceil n/2\rceil\).
	Let \(Q_1\) and \(Q_2\) be two paths in \(\D_K\) having \(u\) as end vertex.
	Decompose \(P\) into two paths \(P_1\) and \(P_2\) having \(u\) as end vertex.
	Let \(R_1 = P_1\cup Q_1\) and \(R_2 = P_1 \cup Q_2\),
	and let \(\D = \big(\D_K\setminus\{Q_1,Q_2\}\big)\cup\{R_1,R_2\}\).
	Clearly, \(\D\) is a path decomposition of \(G\) with \(|\D|\leq\lceil n/2\rceil\).
	Therefore, \(\pn(G)\leq\lceil n/2\rceil\).
\end{proof}

Let \(K\), \(G'\), and \(G = K\cup G'\) be graphs.
We say that \(K\) is a \emph{hanging ESET subgraph} of \(G\) if \(K\) is 
an ESET graph, \(V(K) \cap V(G') = \{u\}\),
and \(u\) is a connection vertex of \(K\).

\begin{lemma}\label{lemma:absorbing-all}
    Let \(G\) be a graph that contains a hanging ESET subgraph \(K\),
    and let \(G'\) be such that \(G = K\cup G'\) and \(V(K)\cap V(G') = \{u\}\).
    Then \(\pn(G)\leq\lceil|V(K)|/2\rceil+\pn(G')-1\).
\end{lemma}

\begin{proof}
    Let \(G\), \(K\), \(G'\), and \(u\) be as in the statement.
Let \(\D'\) be a path decomposition of~\(G'\).
    Let \(P\) be an element of \(\D'\) that contains \(u\)
    and let \(H = K\cup P\).
    By Lemma~\ref{lemma:ESET-absorbing}, we have \(\pn(H)\leq\big\lceil |V(K)|/2\big\rceil\).
    Let \(\D_H\) be a minimum path decomposition of \(H\),
    and note that \(\D_H\cup(\D'\setminus\{P\})\) 
    is a path decomposition of \(G\) 
    such that \(|\D_H\cup(\D'\setminus\{P\})|\leq\big\lceil |V(K)|/2\big\rceil + \pn(G')-1\),
    as desired.
\end{proof}

\section{Graphs with maximum E-degree at most $3$}\label{sec:main-result}

The strategy of the proof of the main theorem of this section is to show that the even subgraph of a minimal counterexample consists of vertex-disjoint triangles, and then proving that these graphs satisfy Gallai's Conjecture.

\begin{theorem}\label{thm:main-theorem-1}
If \(G\) is a connected simple graph on \(n\) vertices 
such that \(\Delta(\EV{G})\leq 3\), 
then either \(G\) is a Gallai graph 
or \(G\) is a SET graph.
\end{theorem}

\begin{proof}Suppose that the statement does not hold, 
	and let \(G\) be a counterexample 
	on \(n\) vertices with \(\Delta(\EV{G})\leq 3\) 
	and which minimizes \(|E(G)|\).
	In what follows, we prove a few claims regarding \(G\).
First, we prove that every hanging ESET subgraph of \(G\) must be connected at a special vertex.
	
	\begin{claim}\label{claim:no-hanging-ESET}
	    Let \(K\) be a hanging ESET subgraph of \(G\),
	    and let \(G'\) be such that \(G = K\cup G'\), and \(V(K)\cap V(G')=\{u\}\).
	    Then \(u\) is an odd vertex of \(G\) and an even vertex of \(G'\).
\end{claim}
	
	\begin{proof}
	    Let \(K\), \(G'\), and \(u\) be as in the statement.
	    Suppose, for a contradiction, that \(u\) has even degree in \(G\)
	    or odd degree in \(G'\).
Then \(\EV{G'}\subseteq\EV{G}\),
	    and hence \(\Delta(\EV{G'})\leq 3\).
	    By the minimality of \(G\), the graph \(G'\) is either a Gallai graph or a SET graph.
	    First, suppose that \(G'\) is a Gallai graph, 
	    i.e., \(\pn(G')\leq\lfloor|V(G')|/2\rfloor\).
	    By Lemma~\ref{lemma:absorbing-all}, it follows that
	    \begin{align*}
	        \pn(G) &\leq \big\lceil|V(K)|/2\big\rceil + \big\lfloor|V(G')|/2\big\rfloor-1\\
	                &\leq \big(|V(K)|+1\big)/2 + |V(G')|/2-1\\
	                &=\big(|V(K)|+|V(G')|-1\big)/2\\
	                &=|V(G)|/2.
	    \end{align*}
        Therefore,  \(\pn(G)\leq \big\lfloor |V(G)|/2\big\rfloor\),
        and \(G\) is a Gallai graph,
        a contradiction.
	    
	    Thus, we may assume that \(G'\) is a SET graph.
	    By Lemma~\ref{lemma:ESET-special-decomposition},
	    \(K\) (resp.  \(G'\)) admits a path decomposition \(\D_K\) (resp.\ \(\D'\)) 
	    such that \(\D_K(u)\geq 2\) and \(|\D_K|\leq\big\lceil|V(K)|/2\big\rceil\) (resp.\ \(\D'(u)\geq 2\)
	    and \(|\D'|\leq \big\lceil |V(G')|/2\big\rceil\)).
	    Let \(P_1\) and \(P_2\) be paths in \(\D_K\) having \(u\) as end vertices,
	    and \(Q_1\) and \(Q_2\) be paths in \(\D'\) having \(u\) as end vertices.
	    Put \(R_1 = P_1\cup Q_1\) and \(R_2=P_2\cup Q_2\),
	    and note that \(\D = \big(\D_K\setminus\{P_1,P_2\}\big)\cup\big(\D'\setminus\{Q_1,Q_2\}\big)\cup\{R_1,R_2\}\)
	    is a path decomposition of \(G\) with cardinality
	    \begin{align*}
	        |\D|	&\leq	\big\lceil|V(K)|/2\big\rceil	+ \big\lceil |V(G')|/2\big\rceil  -2 \\
	            	&\leq	\big(|V(K)|+1\big)/2			+ \big(|V(G')|+1\big)/2-2\\
	            	&\leq	\big(|V(K)|+|V(G')|-2\big)/2 \\
	            	&<		\big(|V(K)|+|V(G')|-1\big)/2 \\
                	&=		|V(G)|/2.
	    \end{align*}
        Therefore,  \(\pn(G)\leq \big\lfloor |V(G)|/2\big\rfloor\),
        and \(G\) is a Gallai graph,
        a contradiction.
	\end{proof}

    Now we use Fan's techniques to prove that \(\EV{G}\) consists of vertex-disjoint triangles.
    First, we prove that no vertex of \(G\) has a unique even neighbor.

    \begin{claim}\label{claim:no-unique-even-neighbor}
        No vertex of \(G\) has exactly one even neighbor.
    \end{claim}

    \begin{proof}
        Suppose, for a contradiction, that \(G\) contains a vertex \(u\) that has precisely one even neighbor, say \(v\), and let \(G'=G \setminus uv\).
        Note that \(v\) has odd degree in \(G'\) and \(u\) has no even neighbor in \(G'\).
        Therefore, \(\Delta(\EV{G'}) \leq 3\).
We claim that \(G'\) is a Gallai graph.
        By Remark~\ref{remark:gallai-components}, it is enough to prove that no component of \(G'\) is a SET graph.
        Indeed, \(G'\) has at most two components, say \(G'_u\) and \(G'_v\), 
        that contain, respectively, \(u\) and \(v\).
        If \(u\) is an odd vertex of \(G'_u\) and \(G'_u\) is a SET graph,
		then \(u\) must be adjacent to at least two even vertices of \(G'_u\), say \(x,y\).
		But, in this case, \(x\) and \(y\) have even degree in \(G\),
		and hence \(u\) has at least three even neighbors in \(G\), namely \(v\), \(x\), and \(y\),
		a contradiction.
        Thus, we may assume that \(u\) is an even vertex of \(G'_u\)
        and an odd vertex of \(G\),
        and hence \(u\) is an isolated even vertex of \(G'\).
        Thus, \(\EV{G'_u}\) is not a triangle,
        and hence \(G'_u\) is not a SET graph.
        Thus, if \(G'\) is connected, i.e., \(G'_u=G'_v\),
        then \(G'\) is a Gallai graph, as desired.
        Thus, we may assume \(G'_u\neq G'_v\).
        In this case, note that, if \(G'_v\) is a SET graph,
        then \(G'_v\) is a hanging ESET subgraph of \(G\) connected at \(v\),
        which is an even vertex of \(G\), a contradiction to Claim~\ref{claim:no-hanging-ESET}.
        Thus, \(G'_u\) and \(G'_v\) are Gallai graphs as desired.

        Let \(\sD'\) be a minimum path decomposition of \(G'\).
        Since \(v\) has odd degree in \(G'\), it follows \(\D'(v) \geq 1\), and since \(u\) has no even neighbor in \(G'\), we have \(|\{x\in N_{G'}(u)\colon \D'(x)=0\}|=0\).
        Thus, by Lemma~\ref{lemma:fan2}, \(uv\) is addible towards \(u\) with respect to \(\D'\).
        Thus, \(\pn(G)\leq |\D'| = \pn(G')\), and hence \(G\) is a Gallai graph, a contradiction.
    \end{proof}
    
\begin{definition}\label{def:fan}
Let \(G\) be a graph, and let \(F\) be a non-empty subgraph of \(G\) with components \(F_1, F_2, \ldots, F_\ell\).
We say that \(F\) is a \emph{Fan subgraph} if, 
for \(j =  2, \ldots, \ell\), the graph \(F_j\) consists of a single edge 
joining even vertices of \(G\),
and \(F_1\) is either the null graph (graph with an empty set of vertices)
or the following hold.

\begin{enumerate}[(i)]
    \item\label{def:fan2} \(F_1\) is a star with center at a vertex \(u\) and at least two leaves \(v_1, v_2, \ldots, v_k\) and, for  \(i \geq 2\), the vertex \(v_i\)  is even in~\(G\);
\item\label{def:fan5} {\(u\) has no even neighbor in \(G \deleset E(F)\);} and
    \item\label{def:fan6} if \(v_1\) is odd in \(G\), then \(u\) is odd in \(G\)
    and each component of \(\EV{G}\) is a triangle.
\end{enumerate}
\end{definition}

\begin{claim}\label{claim:no-fan-subgraph}
  Let \(F\) be a Fan subgraph of \(G\) and let \(G' = G \deleset E(F)\).
  If \(\Delta(\EV{G'}) \leq 3\), 
then \(G'\) is a Gallai graph.
\end{claim}
\begin{proof}
Let \(F\) and \(G'\) be as in the statement.
By Remark~\ref{remark:gallai-components}, it is enough to prove that no component of \(G'\) is a SET graph.
Thus, let \(H'\) be a component of \(G'\),
and suppose, for a contradiction, that \(H'\) is a SET graph.
Let \(E = \{x_1, x_2, x_3\}\) 
and \(O = V(H') \setminus E = \{y_1, \ldots, y_t\}\) 
be, respectively, the set of even and odd vertices of \(H'\).
Since \(G\) is connected, it follows that \(V(H') \cap V(F) \neq \emptyset\).

  Let \(F_1, F_2, \ldots, F_\ell\) be the components of \(F\) as in Definition~\ref{def:fan}
  where \(F_1\) is either the null graph or a star.
First, suppose that \(F_1\) is the null graph.
  Since, by Definition~\ref{def:fan}, for \(j =2,\ldots,\ell\), 
  \(F_j\) consists of a single edge joining even vertices of \(G\),
  every vertex of \(F\) has odd degree in \(G'\),
  which implies that \(V(F)\cap V(H')\subseteq O\).
  Thus, \(x_1,x_2,x_3\) have even degree in \(G\).
  Suppose that there are two vertices, say \(y\) and \(y'\), in \(V(F)\cap V(H')\).
  By the definition of SET graph, \(y\) and \(y'\) are each adjacent to at least two vertices in \(E\).
  Therefore, \(y\) and \(y'\) have a common neighbor, say \(x_1\), in \(E\).
  But then \(x_2,x_3,y,y'\) are four even neighbors of \(x_1\) in \(G\), 
  a contradiction to \(\Delta(\EV{G})\leq 3\).
  Thus, we may assume that there is precisely one vertex, say \(y\), in \(V(F)\cap V(H')\).
  In this case, \(H'\) is a hanging ESET (of type~\ref{def:eset-1}) 
  connected at \(y\), which is an even vertex of \(G\),
  a contradiction to Claim~\ref{claim:no-hanging-ESET}.
  
  Thus, we may assume that  \(F_1\) is a star, and let \(u\) be its center and \(v_1, v_2, \ldots, v_k\) be its leaves.
  By Definition~\ref{def:fan}\eqref{def:fan2}, we may assume that \(v_i\) is even in \(G\) for \(i \geq 2\).
Also, by Definition~\ref{def:fan}, we have \(k \geq 2\) and, for \(j \geq 2\), 
  \(F_j\) consists of a single edge joining even vertices of \(G\).
  This implies that the only vertices of \(F\) that may have even degree in \(G'\) are \(u\) and \(v_1\),
  and hence \(V(F) \cap E \subseteq \{u, v_1\}\).
  {We may assume that \(u\notin O\), otherwise, by the definition of SET graph, \(u\) would have at least two even neighbors in \(G'\), 
  which contradicts Definition~\ref{def:fan}\eqref{def:fan5}.}
  In what follows, we divide the proof into three cases, 
  depending on whether \(\{u,v_1\}\cap E = \emptyset\), \(u\in E\), or \(v_1\in E\).
  Since \(uv_1 \in E(F)\), we have \(uv_1\notin E(H')\),
  and hence the conditions \(u\in E\) and \(v_1\in E\) are disjoint
  because \(\EV{H'}\) is a triangle.
  Let \(L\) be the set of leaves of \(F\), i.e., \(L = V(F) \setminus
  \{u\}\), and let \(L_E = \{v \in L \colon v \text{ has even degree in } G\}\).
Moreover, note that \begin{equation}\label{eq:degree-leaves}
      d_{G'}(w) = d_{G}(w) - 1 \text{ for every vertex } w \in L,
  \end{equation}
  and that every vertex in \(L_E\) has odd degree in \(G'\).

  Suppose that  \(\{u, v_1\} \cap E = \emptyset\), and hence every vertex in \(E\) has even degree in \(G\).
Analogously to the case \(F_1\) is the null graph, if there are two distinct vertices in \(L\cap O\),
then there is a vertex, say \(x_1\), in \(E\) with E-degree at least \(4\) in \(G\),
a contradiction to \(\Delta(\EV{G})\leq 3\).
Thus we may assume that \(|L \cap O| \leq 1\).
 {Thus \(V(H') \cap V(F) = \{y_i\} \subset O\), for some \(i = 1, \ldots, t\), and hence \(H'\) is a hanging ESET subgraph (of type~\ref{def:eset-1}) 
  connected at \(y_i\), which is an even vertex of~\(G\), 
  a contradiction do Claim~\ref{claim:no-hanging-ESET}.}
Hence, we may assume that \(\{u, v_1\} \cap E \neq \emptyset\).

  Suppose that \(u \in E\), and suppose, without loss of generality, that \(u = x_1\).
  Since \(v_1 \notin E\), \(x_2\) and \(x_3\) have even degree in \(G\), a contradiction to Definition~\ref{def:fan}(\ref{def:fan5}).
  Thus, we may assume that \(u\notin E\).

  Now, suppose that \(v_1 \in E\) and suppose, without loss of generality, that \(v_1 = x_1\).
  As noted above, \(u\notin E\), and hence \(x_2\) and \(x_3\) have even degree in \(G\).
Moreover, \(v_1\) has odd degree in \(G\), and hence by Definition~\ref{def:fan}(\ref{def:fan6}), \(u\) is odd in \(G\) and each component of \(EV(G)\) is a triangle.
If \(L \cap O = \emptyset\), then \(H'\) is a hanging ESET subgraph (of type~\ref{def:eset-1})
  connected at \(v_1\), which is an odd vertex of \(G\).
  Let \(G''\) be such that \(G=H'\cup G''\) and \(V(H')\cap V(G'')=\{v_1\}\).
  Since \(v_1\in E\), the vertex \(v_1\) has odd degree in \(G''\),
  a contradiction do Claim~\ref{claim:no-hanging-ESET}.
  Thus, we may assume that there exists a vertex \(y \in L \cap O\), and hence by~\eqref{eq:degree-leaves}, the vertex \(y\) have even degree in \(G\).
  Since every vertex in \(O\) has at least two neighbors in \(E\), we may assume, without loss of generality, that \(yx_2 \in E(G)\).
  Therefore, \(yx_3 \in E(G)\) because \(x_2, x_3, y\) are even in \(G\), \(\{x_2x_3, yx_2\} \subseteq E(G)\), and each component of \(\EV{G}\) is a triangle.
  In particular, \(x_2\) and \(x_3\) are the only even neighbors of \(y\) in \(G\).
  Note that every vertex in \(O\) is adjacent to \(x_2\) or \(x_3\). 
  Thus, if there is a vertex in \(y' \in L \cap O\) such that \(y' \neq y\), then \(y'\) would have even degree in \(G\) and \(x_2, x_3, y, y'\) would belong to the same component of \(\EV{G}\), a contradiction.
  Therefore \(V(H') \cap V(F) = \{x_1, y\}\).
Suppose that \(y\in V(F_j)\), for some \(j \geq 2\), and let \(V(F_j) = \{y,y'\}\).
Then \(y'\) is an even neighbor of \(y\) in \(G\), and hence \(y' \in\{x_2,x_3\}\subseteq E\), a contradiction.
Thus, we have \(y = v_i\), for some \(i \geq 2\), and  \(uy \in E(F_1)\).
  Note that \(\{ux_1, uy\}\) is an edge cut of \(G\), where \(H'\) is a component of \(G-x_1v-yv\).
  In other words, \(H''=H'+ux_1+uy\) is a hanging ESET subgraph (of type~\ref{def:eset-3})
  connected at \(u\), which is an odd vertex of \(G\).
  Let \(G''\) be such that \(G=H''\cup G''\) and \(V(H'')\cap V(G'')=\{u\}\).
  Since \(d_{H''}(u)=2\), the vertex \(u\) has odd degree in \(G''\),
  a contradiction do Claim~\ref{claim:no-hanging-ESET}.
\end{proof}

    \begin{claim}\label{claim:only-triangles}
        Every component of \(\EV{G}\) is a triangle or an isolated vertex.
    \end{claim}

    \begin{proof}
        Suppose, for a contradiction, that \(\EV{G}\) contains a component \(H\) which is neither a triangle nor an isolated vertex.

        \begin{subclaim}
          Every vertex in \(H\) has degree~\(2\).
        \end{subclaim}

        \begin{proof}
            By Claim~\ref{claim:no-unique-even-neighbor}, no vertex of \(H\) has degree \(1\).
            Thus, we prove that \(H\) has maximum degree \(2\).
            Suppose, for a contradiction, that \(H\) has a vertex \(u\) with degree~\(3\), and let \(\{v_1, v_2, v_3\} \subseteq N_H(u)\) be three even neighbors of \(u\).
            Let \(F\) be the subgraph of \(G\) induced by the edges \(uv_1, uv_2, uv_3\),
            and let \(G_1 = G \setminus E(F)\).
            Note that every vertex in \(V(F)\) has odd degree in \(G_1\) and, hence, that \(\EV{G_1} \subset \EV{G}\).
            Thus, it follows that \(\Delta(\EV{G_1}) \leq 3\).
            Note that \(F\) is a star with center at \(u\) and with at least two leaves,
            in which every leaf has even degree in \(G\).
            Moreover, \(u\) has no odd neighbor in \(G_1\).
            Thus, \(F\) is a Fan subgraph.
            By Claim~\ref{claim:no-fan-subgraph}, we have \(\pn(G_1)\leq \floor{n / 2}\).
Let \(\D_1\) be a minimum path decomposition of \(G_1\).
            Since the vertices of \(F\) have odd degree in \(G_1\), it follows that \(\sD_1(v) \geq 1\) for every \(v \in V(F)\).
            By Lemma~\ref{lemma:fan4}, there is a set \(B\subseteq\{uv_1,uv_2,uv_3\}\) addible towards \(u\) with respect to \(\D_1\) and containing \textcolor{black}{at least} two edges.
            Let \(\D_2\) be the \(B\)-transformation of \(\D_1\) towards \(u\),
            and let \(G_2 = G_1+B\).
            {\color{black}        
            If \(|B|=3\), then \(\D_2\) is a path decomposition of \(G\) such that \(|\D_2| = |\D_1| \leq\floor{n/2}\), a contradiction.
            Thus, we may assume \(|B|=2\).
            Suppose, without loss of generality that \(uv_3\notin B\).
Note that \(\D_2(u)\geq 3\), and that, since \(d_{\EV{G}}(v_3) \leq 3\), we have \(d_{\EV{G_2}}(v_3)\leq 2\), because \(u\) is not a neighbor of \(v_3\) in \(G_1\).
            Thus, by Lemma~\ref{lemma:fan2}, \(uv_3\) is addible towards \(v_3\) 
            with respect to \(\D_2\),} but the \(uv_3\)-transformation \(\D\) of \(\D_2\) towards \(v_3\)
            is a path decomposition of \(G\) such that \(|\D_3|\leq\floor{n/2}\),
            a contradiction.
\end{proof}
        
        Suppose that \(H\) is not a triangle, and let \(u\) be a vertex in \(H\).
        Let \(v_1\) and \(v_2\) be the even neighbors of \(u\).
        Since \(H\) is not a triangle, \(v_1\) and \(v_2\) are not adjacent.
        Let \(w\) be an even neighbor of \(v_2\) different from \(u\).  
        Let \(F_1\) be the null graph, \(F_2 = G[\{u, v_1\}]\) and \(F_3=G[\{v_2, w\}]\),
        and let \(F = F_1\cup F_2\cup F_3\).
        Let \(G' = G \setminus E(F)\), and
        note that every vertex in \(V(F)\) has odd degree in \(G'\), 
        and hence \(\EV{G'}\subseteq \EV{G}\).
        Thus \(\Delta(\EV{G'}) \leq 3\).
        Since \(F_2\) and \(F_3\) consist of single edges joining even vertices of \(G\), \(F\) is a Fan subgraph.
        By Claim~\ref{claim:no-fan-subgraph}, we have \(\pn(G') \leq \floor{n / 2}\). 
Let \(\D'\) be a minimum path decomposition of \(G'\).
        Note that \(v_2\) has no even neighbors in \(G'\).
        Thus, by Lemma~\ref{lemma:fan2}, \(v_2w\) is addible towards \(v_2\) with respect to \(\D'\).
        Let \(\D''\) be the \(v_2w\)-transformation of \(\D'\) towards \(v_2\).
        Note that \(v_2\) is the only even neighbor of \(u\) in \(G'+v_2w\),
        but \(\D''(v_2)\geq 2\).
        Thus, by Lemma~\ref{lemma:fan2}, \(uv_1\) is addible towards \(u\) with respect to \(\D''\),
        but the \(uv_1\)-transformation \(\D\) of \(\D''\) towards \(u\) 
        is a path decomposition of \(G\) such that \(|\D|\leq\floor{n/2}\),
        a contradiction.
\end{proof}

    In what follows, if \(x\) is a vertex of odd degree, and \(T\subseteq \EV{G}\) is a triangle containing a neighbor of \(x\), we say that \(T\) is a \emph{triangle neighbor} of \(x\).
Given a vertex $v$ and a triangle \(T\), we denote by \(E_G(u, T)\) the set of edges in \(G\) joining \(v\) to a vertex in \(T\).
    First, we prove that each odd vertex of \(G\) has even neighbors in at most one component of \(\EV{G}\), which implies that 
    every odd vertex of \(G\) has at most one triangle neighbor.
    This implies that two vertices of even degree have a common neighbor if and only if they belong to the same (triangle) component of \(\EV{G}\).
    The proof of the next claim consists in applying Fan's technique on vertices of odd degree.
    For that, let \(T_1,\ldots, T_s\) be the triangles in \(\EV{G}\).
        
    \begin{claim}\label{claim:one-triangle-neighbor}
        If \(d_G(u)\) is odd, then \(u\) has even neighbors in at most one component of \(\EV{G}\).
    \end{claim} 

    \begin{proof}
        Suppose, for a contradiction, that \(u\) has even neighbors in at least two components of \(\EV{G}\).
        Let \(\mathcal{I}\) be the set of even neighbors of \(u\) which are isolated vertices in \(\EV{G}\).
        For \(k=1,2,3\), let \(\mathcal{T}_k\) be indexes of the triangle neighbors of \(u\) containing \(k\) neighbors of \(u\) and let \(\mathcal{T} =\mathcal{T}_1\cup\mathcal{T}_2\cup\mathcal{T}_3\).
For each \(k\in\{1,2,3\}\) and \(j\in\mathcal{T}_k\), let \(V(T_j) = \{x^j_1,x^j_2, x^j_3\}\), where \(x^j_i \in N_G(u)\) for \(i\leq k\).
        In what follows, we remove a set of edges incident to \(u\) and of some triangles neighbors of \(u\).
        For that, for  \(v\in\mathcal{I}\), let \(\varphi(v) = \{uv\}\); for \(j\in\mathcal{T}_1\cup\mathcal{T}_2\), let \(\varphi(T_j) = \{ux^j_{1},x^j_{2}x^j_{3}\}\); and for \(j\in\mathcal{T}_3\), let \(\varphi(T_j) = \{ux^j_{1},ux^j_{2},ux^j_{3}\}\).
        Let \(B_0 =\big(\bigcup_{v\in\mathcal{I}}\varphi(v)\big)  \cup\big(\bigcup_{j\in\mathcal{T}}\varphi(T_j)\big)\) (see Figure~\ref{fig:one-triangle-neighbor-case1}) 
and let \(B_0^u = E_G(u) \cap B_0\).

        \begin{subclaim}\label{subclaim:one-triangle-neighbor}
            \(\mathcal{T}_2\neq\emptyset\) and \(|B^u_0|\) is odd.
        \end{subclaim}
        
        \begin{proof}
            Suppose, for a contradiction, that \(\mathcal{T}_2=\emptyset\) or \(|B^u_0|\) is even.
            Let \(G_0 = G \setminus B_0\). 
            Note that every vertex of \(B_0 \delv u\) has odd degree in \(G_0\),
            and hence if \(d_{G_0}(u)\) is even, then \(u\) is an isolated vertex in \(\EV{G_0}\).
            Thus, \(\Delta(\EV{G_0}) \leq 3\).
Let \(H_1, H_2,\ldots, H_\ell\) be the components of \(B_0\),
            where \(H_1\) is the subgraph of \(G\) edge-induced by \(B^u_0\).
            By definition of \(B_0\), the graphs \(H_2, \ldots, H_\ell\)
            consist of single edges joining even vertices.
            Moreover, \(H_1\) is a star with center at \(u\) with at least two leaves,  since \(u\) has even neighbors in at least two components of \(\EV{G}\).
            In addition, the leaves of \(H_1\) are even vertices of \(G\),
            and every neighbor of \(u\) in \(G_0\) is odd.
            Thus, \(B_0\) is a Fan subgraph of \(G\), and by Claim~\ref{claim:no-fan-subgraph}, \(\pn(G_0) \leq \floor{n / 2}\).
            Let \(\cD_0\) be a minimum path decomposition of \(G_0\).
 
            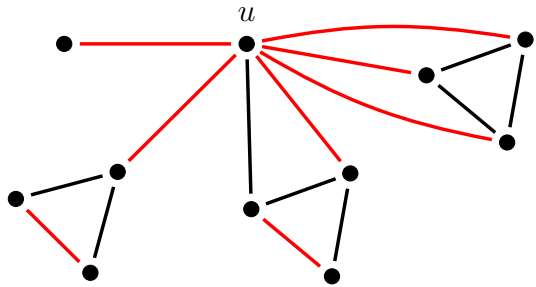
\begin{figure}[h]
                \centering
                \begin{tikzpicture}[scale = 0.8]
	\node (v) [black vertex] 	at (0,0) {};
	
\node (cu) []			at (180:4) {};
	\node (u) [black vertex] 	at ($(cu)+(0:1)$) {};
	\draw[edge,color1] (u) -- (v);
	
\node (c1)  []			at (225:4) {};
	\node (x11) [black vertex] 	at ($(c1)+(45:1)$) {};
	\node (x12) [black vertex] 	at ($(c1)+(165:1)$) {};
	\node (x13) [black vertex] 	at ($(c1)+(285:1)$) {};
	\draw[edge,color1]	(x11) -- (v) (x12) -- (x13);
	\draw[edge]		(x12) -- (x11) -- (x13);
	
\node (c2)  []			at ($(240+50:3.1)$) {};
	\node (x21) [black vertex] 	at ($(c2)+(50:1)$) {};
	\node (x22) [black vertex] 	at ($(c2)+(120+50:1)$) {};
	\node (x23) [black vertex] 	at ($(c2)+(240+50:1)$) {};
	\draw[edge,color1]	(x21) -- (v) (x22) -- (x23);
	\draw[edge]		(v) -- (x22) -- (x21) -- (x23);

\node (c3)  []			at (350:4) {};
	\node (x31) [black vertex] 	at ($(c3)+(120+50:1)$) {};
	\node (x32) [black vertex] 	at ($(c3)+(240+50:1)$) {};
	\node (x33) [black vertex] 	at ($(c3)+(50:1)$) {};
	\draw[edge,color1]	(x31) -- (v) (x32) to [bend left=10] (v) (x33) to [bend right=10] (v);
	\draw[edge]		(x32) -- (x31) -- (x33) -- (x32);

    \node (label_u) at ($(v)+(90:0.5)$) {$u$};

\end{tikzpicture}

                 \caption{Illustration of the Fan subgraph \(B_0\) in the proof of Subclaim~\ref{subclaim:one-triangle-neighbor}.
                         The edges of \(B_0\) are highlighted in red.}\label{fig:one-triangle-neighbor-case1}
            \end{figure}            

            Let \(A_0 =\{x^j_2x^j_3 \colon j \in \mathcal{T}_1 \cup \mathcal{T}_2\}\) and let \(G_1 = G_0 \addeset A_0\).
            Since every vertex in \(B_0 - u \supset V(A_0)\) has odd degree in \(G_0\), it follows that \(\cD_0(v) \geq 1\) for every \(v \in V(B_0) \setminus \{u\})\).
            Note that if \(j \in \mathcal{T}_1\), then \(x^j_2\) and \(x^j_3\) are not adjacent to \(u\) (that may possibly have even degree in \(G_0\)), and hence \(x^j_2\) and \(x^j_3\) have no passing neighbor in \(\cD_0\).
            Moreover, if \(j \in \mathcal{T}_2\), then \(\mathcal{T}_2 \neq \emptyset\) and, by hypothesis, \(|B^u_0|\) is even; hence \(u\) has odd degree in \(G_0\).
            Thus, \(x^j_2\) and \(x^j_3\) have no passing neighbors in \(\cD_0\), for \(j\in\mathcal{T}_1\cup\mathcal{T}_2\).
            Since \(A_0\) is an induced matching in \(G_1\), it follows by
            Lemma~\ref{lemma:induced-matching} that  there is a path decomposition \(\cD_1\) of \(G_1\)
            such that \(|\cD_1| = |\cD_0|\); 
            for every \(j \in \mathcal{T}_1 \cup \mathcal{T}_2\), we have \(\cD_1(x^j_2) = \cD_0(x^j_2) + 1\) and 
            \(\cD_1(x^j_3) = \D_0(x^j_3) - 1\);
            and \(\cD_1(v) = \cD_0(v)\) for every \(v \notin V(A_0)\).

            For \(v \in \big(V(B_0)\setminus \{u\}\big) \setminus V(A_0)\), it follows that \(\cD_1(v) = \cD_0(v) \geq 1\).
            Moreover, for \(j \in\mathcal{T}_1\cup\mathcal{T}_2\), we have \(\cD_1(x^j_2) \geq 2\) and \(x^j_3\) is not a neighbor of \(u\) in \(G\).
            Thus no neighbor of \(u\) in \(G\) is a passing vertex in \(\cD_1\).
            Let \(B_1 = E(G) \setminus E(G_1)\), and note that \(B_1 = B^u_0\).
            By Lemma~\ref{lemma:fan4}, there is an addible set \(A_1 \subseteq B_1\) towards \(u\) with respect to \(\cD_1\) such that \(|A_1| \geq \lceil |B_1|/2 \rceil\).
            Let \(G_2 = G_1 \addeset A_1\),  \(\cD_2\) be the \(A_1\)-transformation of \(\cD_1\) towards \(u\), and \(B_2 = E(G) \setminus E(G_2)\).
            Note that 
            \begin{equation}\label{eq:d2}
            \cD_2(u) \geq \cD_1(u) + \ceil{|B_1| / 2|}.
            \end{equation}

            Let \(B^*_2 = \big\{uv_j \in B_2 \cap E_G(u, T_j) \: j \in \cT_3 \tand |E_{G_2}(u, T_j)| = 2\big\}\).
            Note that \(\cD_2(u) \geq 2 |B^*_2|\).
            Let \(A_2 \subseteq B^*_2\) be a maximal addible set outwards \(u\) with respect to \(\cD_2\), \(G_3 =  G_2 \addeset A_2\), \(\cD_3\) be the \(A_2\)-transformation of \(\cD_2\) outwards \(u\), and \(B_3 = E(G) \setminus E(G_3)\).
            We show that \(A_2 = B^*_2\).
            Suppose, for a contradiction, that \(A_2 \subsetneq B^*_2\) (hence \(B^*_2 \setminus A_2 \neq \emptyset\) and \(B^*_2 \neq \emptyset\)).
            Since \(B^*_2 \neq \emptyset\), it follows that \(\cT_3 \neq \emptyset\), and since \(u\) has neighbors in at least two components of \(EV(G)\) in \(G\), it follows that \(|B_1| = |B^u_0| \geq 4\).
            We claim that \(\D_2(u)\geq 3\).
            This is clear from Equation~\eqref{eq:d2} if \(|B_1|\) is odd.
            If \(|B_1|\) is even, then \(u\) has odd degree in \(G_1\), and hence \(\cD_1(u) \geq 1\).
            By Equation~\eqref{eq:d2}, we have \(\cD_2(u) \geq 3\).
            Let \(uv_j \in B^*_2 \setminus A_2\).
            First, we show that \(|A_2|\geq 1\).
            For that note that the only possible passing neighbors of \(v_j\) in \(\cD_2\) are the vertices in \(V(T_j) \setminus \{v_j\}\).
            Since \(\D_2(u) \geq 3\), it follows by  Lemma~\ref{lemma:fan2} that \(\{uv_j\}\) is an addible set towards \(v\) (i.e. outwards \(u\)) in \(\D_2\) and, therefore, \(|A_2| \geq 1\).  
            Now, since \(A_2 \subsetneq B^*_2\), we have,  \(|B^*_2 \setminus A_2| \geq 1\), and hence,
            \[
                \D_3(u) = \D_2(u) - |A_2| \geq 2 |B^*_2| - |A_2| = 2(|A_2| + |B^*_2 \setminus A_2|) - |A_2| \geq 3.
            \]
            By Lemma~\ref{lemma:fan2}, \(uv_j\) is addible towards \(v_j\) (i.e. outwards \(u\)) with respect to \(\D_3\), and hence, by Remark~\ref{remark:addible-concatenation}, the set \(A_2 \cup \{uv_j\}\subseteq B_2^*\) is an addible set outwards \(u\) with respect to \(\D_2\), a contradiction to the maximality of \(A_2\).
            Therefore, \(A_2 = B^*_2\).

            Let \(A_3 \subseteq B_3\) be a maximum addible set outwards \(u\) with respect to \(D_3\), \(G_4 = G_3 \addeset A_3\), \(D_4\) be the \(A_3\)-transformation outwards \(u\) with respect to \(D_3\), and \(B_4 = E(G) \setminus E(G_4)\).
            In what follows, we show that \(B_4 = \emptyset\), and hence \(G = G_4\).
            Suppose, for a contradiction, that \(B_4 \neq \emptyset\).
            By Lemma~\ref{lemma:extra-edges}, \(\D_4(u) > |B^u_0| - |A_1| - |A_2| -|A_3| = |B_4|\), and since \(|B_4|\geq 1\), it follows that \(\D_4(u) \geq 2\).
            Let \(uv \in B_4\), and note that \(v\) has at most one passing neighbor with respect to \(\D_4\) in \(G_4\):
            \begin{itemize}
                \item if \(v \in \mathcal{I}\), then every neighbor of \(v\) in \(G_4\) has odd degree, and hence \(v\) has no passing neighbor with respect to \(\cD_4\);
                \item if \(v = x^j_i\) for some \(j \in \mathcal{T}_1 \cup \mathcal{T}_2\), then \(i \leq 2\), and the only possible passing neighbor of \(x^j_i\) is \(x^j_3\); and
                \item if \(v = x^j_i\) for some \(j \in \mathcal{T}_3\), then \(|E_{G_2}(u, T_j)| \leq 1\),  otherwise \(uv \in A_2\). Thus there is at most one passing vertex in \(T_j\) with respect to \(\cD_4\) and, therefore, \(x^j_i\) has at most one passing neighbor with respect to \(\cD_4\).
            \end{itemize}
            Since \(\D_4(u) \geq 2\),  by Lemma~\ref{lemma:fan2}, \(uv\) is addible towards \(v\) (i.e. outwards \(u\)) with respect to \(\D_4\), and hence by Remark~\ref{remark:addible-concatenation}, \(A_3 \cup \{uv\}\) is an addible set outwards \(u\) with respect to \(D_3\), a contradiction to the choice of \(A_3\).
            Therefore, \(B_4 = \emptyset\) and \(G_4 = G\).
        \end{proof}
        
        By Subclaim~\ref{subclaim:one-triangle-neighbor}, we have \(\cT_2 \neq \emptyset\) and \(|B^u_0|\) is odd.
        Choose \(z \in \mathcal{T}_2\), and define \(\varphi'(T_{z}) = \{ux^{z}_1,ux^{z}_2\}\).
        Let 
        \(R_0 = \big(\bigcup_{v\in\mathcal{I}}\varphi(v)\big)\cup\big(\bigcup_{j\in\mathcal{T}\setminus\{z\}}\varphi(T_j)\big) \cup \varphi'(T_{z}) \text{ (see Figure~\ref{fig:one-triangle-neighbor-case2})}\),
        and let \(R^u_0 = R_0 \cap E_{G}(u)\).
        Note that \(|R^u_0| = |B^u_0|+1\), and hence \(|R^u_0|\) is even.
        Let \(G_0 = G \setminus R_0\),  and note that every vertex in \(V(R_0) \setminus \{x^z_3\}\) has odd degree in \(G_0\).
        Since no odd vertex of \(G\) is an even vertex of \(G_0\), 
        we have \(\EV{G_0}\subseteq\EV{G}\),
        and hence \(\Delta(\EV{G_0})\leq 3\).
        Let \(H_1,\ldots,H_\ell\) be the components of \(R_0\),
        where \(H_1\) is the subgraph of \(G\) induced by the edges in \(R^u_0\).
        Note that \(H_1\) is a star with center at \(u\) and at least four leaves, i.e.,
        \(|R^u_0| \geq 4\), since \(u\) has even neighbors in at least two components of \(\EV{G}\),
        \(\cT_2 \neq \emptyset\), and \(|R^u_0|\) is even.
        Also, all the leaves of \(R^u_0\) are even in \(G\), and \(u\) has no even neighbors in \(G_0\).
        Moreover, note that \(H_2, \ldots, H_\ell\) consist of single edges joining even vertices.
        Therefore, \(R_0\) is a Fan subgraph,
        and by Claim~\ref{claim:no-fan-subgraph}, \(\pn(G_0)\leq\floor{n / 2}\).
        Let \(\D_0\) be a minimum path decomposition of \(G_0\).
        
        \begin{figure}[ht]
            \centering
            \begin{tikzpicture}[scale = 0.8]
	\node (v) [black vertex] 	at (0,0) {};
	
\node (cu) []			at (180:4) {};
	\node (u) [black vertex] 	at ($(cu)+(0:1)$) {};
	\draw[edge,color1] (u) -- (v);
	
\node (c1)  []			at (200:4) {};
	\node (x11) [black vertex] 	at ($(c1)+(20:1)$) {};
	\node (x12) [black vertex] 	at ($(c1)+(140:1)$) {};
	\node (x13) [black vertex] 	at ($(c1)+(260:1)$) {};
	\draw[edge,color1]	(x11) -- (v) (x12) -- (x13);
	\draw[edge]		(x12) -- (x11) -- (x13);
	
\node (c2)  []			at ($(235:3.1)$) {};
	\node (x21) [black vertex] 	at ($(c2)+(-5:1)$) {};
	\node (x22) [black vertex] 	at ($(c2)+(120-5:1)$) {};
	\node (x23) [black vertex] 	at ($(c2)+(240-5:1)$) {};
	\draw[edge,color1]	(x21) -- (v) (x22) -- (x23);
	\draw[edge]		(v) -- (x22) -- (x21) -- (x23);
	
\node (c2)  []			at ($(240+50:3.1)$) {};
    \node (label_x1) at ($(240+51:3.2)$) {$T_{z}$};
	\node (x21) [black vertex] 	at ($(c2)+(50:1)$) {};
	\node (x22) [black vertex] 	at ($(c2)+(120+50:1)$) {};
	\node (x23) [black vertex] 	at ($(c2)+(240+50:1)$) {};
	\draw[edge,color1]	(x21) -- (v) (x22) -- (v);
	\draw[edge]		(x23) -- (x22) -- (x21) -- (x23);

\node (c3)  []			at (350:4) {};
	\node (x31) [black vertex] 	at ($(c3)+(120+50:1)$) {};
	\node (x32) [black vertex] 	at ($(c3)+(240+50:1)$) {};
	\node (x33) [black vertex] 	at ($(c3)+(50:1)$) {};
	\draw[edge,color1]	(x31) -- (v) (x32) to [bend left=10] (v) (x33) to [bend right=10] (v);
	\draw[edge]		(x32) -- (x31) -- (x33) -- (x32);

\node (label_u) at ($(v)+(90:0.5)$) {$u$};
\end{tikzpicture}

             \caption{Illustration of the Fan subgraph \(R_0\) in the proof of Claim~\ref{claim:one-triangle-neighbor}.
                    The edges of \(R_0\) are highlighted in red.}
            \label{fig:one-triangle-neighbor-case2}
        \end{figure}
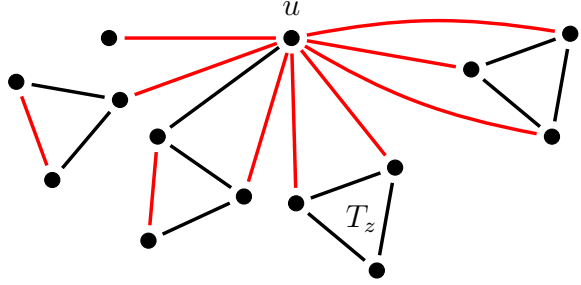        
        
        Note that every vertex  \(v \in V(R_0)\) has odd degree in \(G_0\), and hence \(\cD_0(v) \geq 1\).
        In particular, \(\cD_0(u) \geq 1\).
Also, note that every vertex \(v \in V(R_0) \setminus \{x^z_1, x^z_2\}\) has no neighbor with even degree in \(G_0\), and hence no passing neighbor in \(\cD_0\).
        On the other hand, \(x^{z}_1\) and \(x^{z}_2\) has only one possible passing neighbor in \(\cD_0\), namely the vertex \(x^{z}_3\) which has even degree in \(G_0\).
Let \(A_0 =\big\{x^z_2 x^z_3 \colon j \in \mathcal{T}_1 \cup \mathcal{T}_2 \setminus \{z\} \big\}\), \(G_1 = G_0 \addeset A_0\), and \(R_1 = E(G) \setminus E(G_1)\) (remark that \(R_1 = R^u_0\)). 
        Note that \(A_0\) is an induced matching in \(G_1\).
        If \(A_0 \neq \emptyset\), then let \(\D_1\) be the path decomposition given by Lemma~\ref{lemma:induced-matching}; otherwise let \(\D_1 = \D_0\).
        In either case, it follows that \(\D_1\) is a path decomposition of \(G_1\) such that
        \begin{itemize}
            \item \(|\D_1| = |\D_0|\); 
            \item for all \(j\in\mathcal{T}_1\cup\mathcal{T}_2 \setminus\{z\}\), we have \(\D_1(x^j_2) = \D_0(x^j_2) + 1\) and \(\D_1(x^j_3) = \D_0(x^j_3) - 1\); and
            \item \(\D_1(v) = \D_0(v)\) for every \(v \notin V(A_0)\). 
        \end{itemize}
        
        Note that, for \(j\in\mathcal{T}_1\cup\mathcal{T}_2\setminus\{z\}\), we have \(\D_1(x^j_2) \geq 2\), and for \(j\in\mathcal{T}_1\cup\mathcal{T}_2\), the vertex \(x^j_3\) is not adjacent to \(u\) in \(G\).
        Therefore, no neighbor of \(u\) in \(G\) is a passing vertex in \(\D_1\).
        By Lemma~\ref{lemma:fan4}, there is an addible set \(A_1 \subseteq R_1\) towards \(u\) with respect to \(\D_1\) such that \(|A_1|\geq \ceil{|R_1|/2}\).
        Let \(G_2 = G_1 \addeset A_1\), \(\D_2\) be the \(A_1\)-transformation of \(\D_1\) towards \(u\), and \(R_2 = E(G) \setminus E(G_2)\).
        If \(R_2 = \emptyset\), then \(G = G_2\), and \(\cD_2\) is a path decomposition of \(G\) such that \(|\cD_2| = |\cD_1| \leq \floor{n/2}\), a contradiction to the choice of \(G\).
        Thus, we may assume that \(R_2 \neq \emptyset\).

        By the construction of \(\cD_2\), it follows that
        \begin{equation}
        \cD_2(u) = \cD_1(u) + |A_1|.
        \end{equation}
        Let \(R^*_2\) be the set of edges \(uv_j \in R_2\), where  \( j \in \cT_3 \cup \{z\}\), such that 
        \(uv_j\) is the only edge in \(E(G)\) joining \(u\) to a vertex of \(T_j\) which does not belong to \(G_2\), i.e., 
        \[R^*_2 = \{uv_j \in R_2 \cap E_{G}(u, T_j) \:  j \in \cT_3 \cup \{z\} \tand |R_2 \cap E_{G}(u, T_j)| = 1\}.\]
        For every edge \(uv_j \in R^*_2\), it follows that the vertex \(v_j\) has at most two passing neighbors in \(\cD_2\),
        namely, the vertices in \(V(T_j) \setminus \{v_j\}\), which have even degree in \(G_2\). 
        Let \(A_2 \subseteq R^*_2\) be a maximum addible set outwards \(u\) with respect to \(\cD_2\). 
        Let \(G_3 = G_2 \addeset A_2\), \(\cD_3\) be the \(A_2\)-transformation outwards \(u\) with respect to \(\cD_2\), and \(R_3 = E(G) \setminus E(G_3)\).

        Now we show that \(A_2 = R^*_2\).
        Suppose, for a contradiction, that \(A_2 \subsetneq R^*_2\), and let \(uv_t \in R^*_2 \setminus A_2\).
        Note that \(\cD_2(u) = \cD_1(u) + |A_1| \geq 1 + \ceil{|R_1| / 2} =  1 + \ceil{|R^u_0| / 2} \geq 3\).
        Thus, by Lemma~\ref{lemma:fan2}, \(\{uv_t\}\) is an addible set outwards \(u\) with respect to \(\cD_2\), and hence \(|A_2| \geq 1\).
        By the construction of \(\cD_3\), it follows that \(\cD_3(u) = \cD_2(u) - |A_2|\).
        Now we show that \(\cD_2(u) \geq 2 |R^*_2|\).
        This is clear if \(uv_z \notin R^*_2\); thus suppose that \(uv_z \in R^*_2\), and hence
        \[\cD_2(u) = \cD_1(u) + |A_1| \geq 1 + 2 (|R^*_2| - 1)  + 1  = 2 |R^*_2|.\]
        Therefore, 
        \[
        \cD_3(u) = \cD_2(u) - |A_2|  \geq 2 |R^*_2| - |A_2| = 2( |A_2| + |R^*_2 \setminus A_2|) - |A_2| = |A_2| + 2 |R^*_2 \setminus A_2| \geq 3.
        \]
        By Remark~\ref{remark:addible-concatenation}, \(A_2 \cup \{uv_t\}\) is an addible set outwards \(u\) with respect to \(\cD_2\), a contradiction to the choice of \(A_2\).

        If \(R_3 = \emptyset\), then \(G = G_3\) and \(\cD_3\) is a path decomposition of \(G\) such that \(|\cD_3| = |\cD_2| \leq \floor{n /2}\), a contradiction to the choice of \(G\).
        Thus, we may assume that \(R_3 \neq \emptyset\).
        Let \(A_3 \subseteq R_3\) be a maximum addible set outwards \(u\) with respect to \(\cD_3\).
        Let \(G_4 = G_3 \addeset A_3\), \(\cD_4\) be the \(A_3\)-transformation outwards \(u\) with respect to \(\cD_3\), and \(R_4 = E(G) \setminus E(G_4)\).
        Now, we show that \(A_3 = R_3\).
        Suppose, for a contradiction, that \(A_3 \subsetneq R_3\), and let \(uv_t \in R_3 \setminus A_3\).
        By Lemma~\ref{lemma:extra-edges}, we have \(\D_4(u) > |R_3 \setminus A_3| \geq 1\),
        and hence \(\D_4(u)\geq 2\).
		We claim that \(v_t\) has at most one passing neighbor with respect to \(\D_4\).
        Indeed, if \(v_t \in \cI\), then every neighbor of \(v_t\) in \(G_4\) has odd degree, and hence \(v_t\) has no passing neighbor;
        if \(v_t = x^j_1\) for some \(j\in\mathcal{T}_1\cup\mathcal{T}_2\setminus\{z\}\),
        then the only passing neighbor of \(x^j_1\) is possibly \(x^j_3\);
        if \(v_t = x^j_i\) for some \(j\in\mathcal{T}_3\) and \(i\in\{1,2,3\}\),
        then, \(|E_{G_2}(u, T_t)| \leq 1\), otherwise \(uv_t \in A_2\). Thus there is at most one passing vertex in \(T_t\) with respect to \(\D_4\) and, therefore, \(v_t\) has at most one passing neighbor with respect to \(\D_4\); finally, if \(v_t = x^z_i\) for some \(i\in\{1,2\}\),
        then we must have \(|E_{G_2}(u, T_z)| = 0\), otherwise \(uv_t \in A_2\).
        Thus, the unique passing vertex in  \(T_{z}\) is \(x^z_3\), and hence \(v_t\) has at most one passing neighbor in \(\cD_4\).
        Since \(\D_4(u) \geq 2\), by Lemma~\ref{lemma:fan2},         
        \(uv_t\) is addible towards \(v_t\) with respect to \(\D_4\), hence by Remark~\ref{remark:addible-concatenation},  \(A_3 \cup \{uv_t\}\) is an addible set 
        outwards \(u\) with respect to \(\cD_3\), a contradiction to the choice of \(A_3\).
        Thus, \(A_3 =  R_3 =  E(G) \setminus E(G_3)\) and \(G = G_4\).
        Therefore, \(\cD_4\) is a path decomposition of \(G\) such that \(|\cD_4| = |\cD_3| \leq \floor{n /2}\), a contradiction to the choice of \(G\).
    \end{proof}
        
    Now, we can prove that \(\EV{G}\) contains no isolated vertex.
    Indeed, let \(x\) be an isolated vertex of \(\EV{G}\), and let \(y\) be any neighbor of \(x\).
    Since \(x\) is isolated, \(y\) has odd degree.
    By Claim~\ref{claim:one-triangle-neighbor}, \(x\) is the unique even neighbor of \(y\), a contradiction to Claim~\ref{claim:no-unique-even-neighbor}.

    Given a vertex \(u\) of \(G\) that has a triangle neighbor \(T\), we say that \(u\) is a \emph{full vertex} if every vertex of \(T\) is a neighbor of \(u\).
        
    \begin{claim}\label{claim:full-vertices1}
        Let \(u\) be a vertex of \(G\) that has a triangle neighbor.
        If \(u\) has an odd neighbor that has no even neighbor, then \(u\) is a full vertex.
    \end{claim}

    \begin{proof}
        Let \(u\) be as in the statement, let \(T\) be its triangle neighbor, and let \(v\) be an odd neighbor of \(u\) that has no even neighbor.
        Let \(V(T) = \{x,y,z\}\).
        Suppose, for a contradiction, that \(u\) has at most two neighbors in \(T\).
        Thus,  we may suppose, without loss of generality, that \(z\notin N(u)\).
        Note that, by Claims~\ref{claim:no-unique-even-neighbor} and~\ref{claim:one-triangle-neighbor}, \(u\) must be adjacent to \(x\) and \(y\). 
        Let \(F_1\) be the subgraph of \(G\) induced by the edges \(uv\) and \(ux\),
        and \(F_2\) be the subgraph of \(G\) induced by \(yz\),
        and let \(F = F_1\cup F_2\).
Let \(G_0 = G \setminus E(F)\)
        and note that \(d_{G_0}(w)\) is odd for every \(w\in V(T)\cup\{u\}\),
        and since \(v\) has no even neighbors in \(G\),
        \(v\) is an isolated vertex of \(\EV{G_0}\).
Therefore, \(\EV{G_0}-v\subseteq\EV{G}\), and hence \(\Delta(\EV{G_0})\leq 3\).
We claim that \(F\) is a Fan subgraph.
        Indeed, \(F_2\) consists of a single edge joining even vertices of \(G\).
        Moreover, \(F_1\) is a star with center at \(u\) and two leaves \(v_1=v\) and \(v_2=x\),
        where \(d_G(v_2)\) is even,
        and hence Definition~\ref{def:fan}\eqref{def:fan2} holds.
        Also, every neighbor of \(u\) has odd degree in \(G_0\), 
        and hence  Definition~\ref{def:fan}\eqref{def:fan5} holds.
        Finally, \(v_1\) is an odd vertex in \(G\), and \(u\) has odd degree in \(G\),
        and every component of \(\EV{G}\) is a triangle,
        which verifies Definition~\ref{def:fan}\eqref{def:fan6}.
        By Claim~\ref{claim:no-fan-subgraph}, \(\pn(G_0)\leq \floor{n / 2}\).
Let \(\D_0\) be a minimum path decomposition of \(G_0\).
In what follows, we restore the edges \(yz\), \(xu\), and \(uv\), in this order  (see Figure~\ref{fig:full-vertex1}).
        Note that every neighbor of \(y\) has odd degree in \(G_0\), thus, by Lemma~\ref{lemma:fan2}, \(yz\) is addible towards \(y\) with respect to \(\D_0\).
        Let \(\D_1\) be the \(yz\)-transformation of \(\D_0\) towards \(y\), and note that every neighbor of \(u\) in \(G_1=G_0+yz\) has odd degree, except for \(y\), but we have \(\D_1(y)\geq 2\).
        Thus, by Lemma~\ref{lemma:fan2}, \(xu\) is addible towards \(u\) with respect to \(\D_1\).
        Let \(\D_2\) be the \(xu\)-transformation of \(\D_1\) towards \(u\), and note that every neighbor of \(v\) in \(G_2 = G_1+xu\) has odd degree (note that \(u\) is not a neighbor of \(v\) in \(G_2\)).
        Again, by Lemma~\ref{lemma:fan2}, \(uv\) is addible towards \(v\) with respect to \(\D_2\).
        Let \(\D\) be the \(uv\)-transformation of \(\D_2\) towards \(v\).
        Therefore, \(|\D| = |\D_2| = |\D_1|=|\D_0|\leq\floor{n/2}\), a contradiction. 
    \end{proof}

    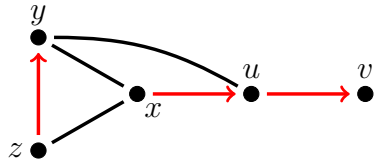
\begin{figure}[h]
        \centering
        \begin{tikzpicture}[scale = 1.5]
	\node (v) [black vertex] 	at (0,0) {};
	\node (u) [black vertex]	at (1,0) {};
	
\node (x) [black vertex]	at (-1,0) {};
	\node (y) [black vertex]	at ($(x) + (150:1)$) {};
	\node (z) [black vertex] 	at ($(x) + (-150:1)$) {};

\node () []	at ($(x) + (-45:0.2)$) {$x$};
	\node () []	at ($(y) + (90:0.2)$) {$y$};
	\node () [] 	at ($(z) + (180:0.2)$) {$z$};
	\node () []	at ($(v) + (90:0.2)$)  {$u$};
	\node () []	at ($(u) + (90:0.2)$)  {$v$};
	
	\draw[edge,color1,->] (v) -- (u);
	\draw[edge,color1,->] (x) -- (v);
	\draw[edge,color1,->] (z) -- (y);
	\draw[edge]	   (y) -- (x) -- (z);
	\draw[edge] (v) to [bend right=15] (y);
\end{tikzpicture}

         \caption{Illustration of the Fan subgraph \(F\) in the proof of Claim~\ref{claim:full-vertices1}. The edges of \(F\) are highlighted in red. A directed edge \(\Vec{ab}\) indicates the use of Lemma~\ref{lemma:fan2} adding the edge \(ab\) towards \(b\), following the ordering $yz, xu, uv$.}
        \label{fig:full-vertex1}
    \end{figure}

    \begin{claim}\label{claim:full-vertices2}
        Let \(u\) be a vertex of \(G\) that has a triangle neighbor.
        Then every odd neighbor \(u\) has an even neighbor.
    \end{claim}

    \begin{proof}
        Let \(u\) be as in the statement, let \(T\) be its triangle neighbor.
        Suppose, for a contradiction, that \(u\) has an odd neighbor \(v\) that has no even neighbor.
        By Claim~\ref{claim:full-vertices1}, every vertex of \(T\) is a neighbor of \(u\).
        Let \(V(T) = \{x,y,z\}\) and let \(S = \{ux,uy,uz\}\).
        Let \(G_0 = G-uv-S\), and note that \(d_{G_0}(w)\) is odd for every \(w\in V(T)\cup\{u\}\),
        and since \(v\) has no even neighbors in \(G\),
        \(v\) is an isolated vertex of \(\EV{G_0}\).
Therefore, \(\EV{G_0}-v\subseteq\EV{G}\), and hence \(\Delta(\EV{G_0})\leq 3\).
        Let \(F=F_1\) be the subgraph of \(G\) induced by the edges in \(S\cup\{uv\}\).
        We claim that \(F\) is a Fan subgraph.
        Indeed, \(F_1\) is a star with center at \(u\) and four leaves \(v_1=v\) and \(v_2=x\), \(v_3=y\), \(v_4=z\)
        where \(d_G(v_i)\) is even, for \(i\in\{2,3,4\}\),
        and hence Definition~\ref{def:fan}\eqref{def:fan2} holds.
        Also, every neighbor of \(u\) has odd degree in \(G_0\), 
        and hence  Definition~\ref{def:fan}\eqref{def:fan5} holds.
        Finally, \(v_1\) is an odd vertex in \(G\), and \(u\) has odd degree in \(G\),
        and every component of \(\EV{G}\) is a triangle,
        which verifies Definition~\ref{def:fan}\eqref{def:fan6}.
        By Claim~\ref{claim:no-fan-subgraph}, \(\pn(G_0)\leq \floor{n / 2}\).
Let \(\D_0\) be a minimum path decomposition of \(G_0\).
Note that no neighbor of \(u\) in \(G_0+S=G-uv\) is a passing vertex in \(\D_0\).
        By Lemma~\ref{lemma:fan4}, there is a \(B\subseteq S\) such that \(|B|\geq \lceil |S|/2\rceil\) and \(B\) is addible towards \(u\) with respect to \(\D_0\).
        Let \(\D_1\) be the \(B\)-transformation of \(\D_0\) towards \(u\).
        We have \(\D_1(u) \geq 1 + \lceil|S|/2\rceil\geq 3\).
        {\color{black}
        Note that \(S\setminus B\) contains at most one edge.
        In what follows, we obtain a decomposition \(\D_2\) of \(G_2=G_0+S=G-uv\) such that \(\D_2(u)\geq 2\).
        If \(S\setminus B =\emptyset\), then \(\D_2 = \D_1\) is the desired decomposition (see Figure~\ref{fig:full-vertex2-1}).
        If \(S\setminus B\neq \emptyset\), then suppose \(uz\in S\setminus B\) and put \(G_1 = G_0+B\).
        Note that the only passing neighbors of \(z\) are possibly \(x\) and \(y\), and hence, by Lemma~\ref{lemma:fan2}, \(uz\) is addible towards \(z\) with respect to \(\D_1\).}
Then, the \(uv\)-transformation \(\D_2\) of \(\D_1\) towards \(z\) is the desired decomposition (see Figure~\ref{fig:full-vertex2-2}).
        Finally, note that every neighbor of \(v\) in \(G_2\) has odd degree, and hence, by Lemma~\ref{lemma:fan2}, \(uv\) is addible towards \(v\) with respect to \(\D_2\).
        Then, the \(uv\)-transformation \(\D\) of \(\D_2\) towards \(v\) is a decomposition of \(G\) such that \(|\D|\leq\floor{n / 2}\), a contradiction.
    \end{proof}

    \begin{figure}[h]
        \centering
        \begin{subfigure}{.45\textwidth}
            \centering
            \begin{tikzpicture}[scale = 1.5]
	\node (v) [black vertex] 	at (0,0) {};
	\node (u) [black vertex]	at (1,0) {};
	
\node (x) [black vertex]	at (-1,0) {};
	\node (y) [black vertex]	at ($(x) + (150:1)$) {};
	\node (z) [black vertex] 	at ($(x) + (-150:1)$) {};

\node () []	at ($(x) + (-45:0.2)$) {$x$};
	\node () []	at ($(y) + (90:0.2)$) {$y$};
	\node () [] 	at ($(z) + (180:0.2)$) {$z$};
	\node () []	at ($(u) + (90:0.2)$)  {$v$};
	\node () []	at ($(v) + (90:0.2)$)  {$u$};
	
	\draw[edge,color1,->] (v) -- (u);
	\draw[edge,color1,->] (x) -- (v);
	\draw[edge,color1,->] (y) to [bend left=15] (v);
	\draw[edge,color1,->] (z) to [bend right=15] (v);
	\draw[edge] (z) -- (y);
	\draw[edge]	   (y) -- (x) -- (z);
\end{tikzpicture}

             \caption{}
            \label{fig:full-vertex2-1}
        \end{subfigure}
        \begin{subfigure}{.45\textwidth}
            \centering
            \begin{tikzpicture}[scale = 1.5]
	\node (v) [black vertex] 	at (0,0) {};
	\node (u) [black vertex]	at (1,0) {};
	
\node (x) [black vertex]	at (-1,0) {};
	\node (y) [black vertex]	at ($(x) + (150:1)$) {};
	\node (z) [black vertex] 	at ($(x) + (-150:1)$) {};

\node () []	at ($(x) + (-45:0.2)$) {$x$};
	\node () []	at ($(y) + (90:0.2)$) {$y$};
	\node () [] 	at ($(z) + (180:0.2)$) {$z$};
	\node () []	at ($(v) + (90:0.2)$)  {$u$};
	\node () []	at ($(u) + (90:0.2)$)  {$v$};
	
	\draw[edge,color1,->] (v) -- (u);
	\draw[edge,color1,->] (x) -- (v);
	\draw[edge,color1,->] (y) to [bend left=15] (v);
	\draw[edge,color1,->] (v) to [bend left=15] (z);
	\draw[edge] (z) -- (y);
	\draw[edge]	   (y) -- (x) -- (z);
\end{tikzpicture}

             \caption{}
            \label{fig:full-vertex2-2}
        \end{subfigure}
        \caption{Illustration of the Fan subgraph \(F\) in the proof of  Claim~\ref{claim:full-vertices2}.
                    The edges of \(F\) are highlighted in red.
                    A directed edge \(\Vec{ab}\) indicates the use of Lemmas~\ref{lemma:fan2} and~\ref{lemma:fan4} adding the edge \(ab\) towards \(b\).}
        \label{fig:full-vertex2}
    \end{figure}
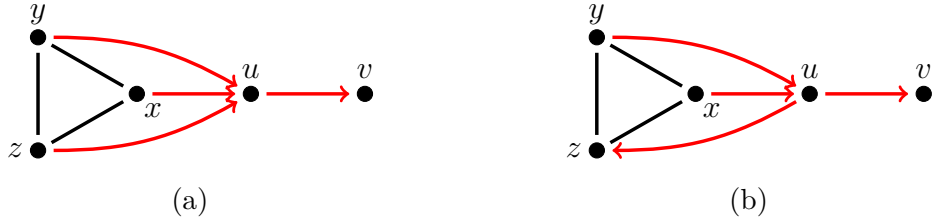
    
     \begin{claim}\label{claim:no-free-odd-vertex}
        Every odd vertex of \(G\) has a triangle neighbor.
    \end{claim}
    
    \begin{proof}
        Let \(v\) be an odd vertex, and let \(P\) a shortest path
        joining \(v\) to an even vertex, say \(x\).
        Let \(u\) be the neighbor of \(x\) in \(P\), and let \(w\) be the neighbor of \(u\) different from \(x\) in \(P\).
        By Claim~\ref{claim:full-vertices2}, \(w\) has an even neighbor, say \(y\).
        If \(w=v\), then the claim follow;
        and if \(w\neq v\), the path \((P\setminus \{xu,uw\})\cup \{wy\}\)
        is a path joining \(v\) to an even vertex and that is shorter than \(P\),
        a contradiction.
    \end{proof}
    
	\begin{claim}\label{claim:no-remaining-only-odds}
	    Let \(F\) be a non-empty subgraph of \(G\)
	    and let \(G' = G\setminus E(F)\).
Suppose that for every odd vertex \(u\) of \(G\) in \(V(F)\)
		the set \(N_F(u)\) contains an even vertex of \(G\).
		Then \(\pn(G')\leq\floor{n/2}\).
	\end{claim}
	
	\begin{proof}
		By Remark~\ref{remark:gallai-components},
		it is enough to prove that no component of \(G'\) is a SET graph.
Indeed, suppose that some component, say \(H'\) of \(G'\) is a SET graph,
        and let \(E\) and \(O\) be, respectively, the sets
        of even and odd vertices of \(H'\).
        Also, let \(E_F\) and \(O_F\) be, respectively, the even and odd vertices of \(G\) in \(V(F)\).
        Since every vertex of \(F\) has odd degree in \(G'\),
        we have \(\EV{G'}\subseteq\EV{G}\),
        and hence \(\Delta(\EV{G'})\leq 3\).
        Moreover, the vertices of \(F\) are not in \(E\).
Thus, the vertices in \(E\) have even degree in \(G\).
        Since every vertex of \(O\) has a neighbor in \(E\),
        if \(O\cap E_F\neq\emptyset\), 
        then \(\EV{G}\) contains a component different from a triangle.
        Now, let \(u\in O\cap O_F\).
        By hypothesis, if \(u\in O_F\),
        then \(F\) contains an edge \(uv\) such that \(v\) is an even vertex of \(G\).
        Thus, \(u\) is an odd vertex of \(G\) that has neighbors
        in more than one component of \(\EV{G}\),
        a contradiction to Claim~\ref{claim:one-triangle-neighbor}.	    
\end{proof}
        
    \begin{claim}\label{claim:full-vertices3}
        If \(u\) and \(v\) are two adjacent odd vertices having neighbors in two distinct even components, then \(u\) and \(v\) are a full vertices.
    \end{claim}

    \begin{proof}
        Let \(u\) and \(v\) be as in the statement, and let \(T_u\) and \(T_v\) be the triangle neighbors of \(u\) and \(v\), respectively.
        Let \(V(T_u) = \{a,b,c\}\) and \(V(T_v) = \{x,y,z\}\).
        Suppose, for a contradiction, that \(v\) has at most two neighbors in \(T_v\).    
        Thus, we may suppose, without loss of generality, that \(z\notin N(v)\).
        By Claims~\ref{claim:no-unique-even-neighbor} and~\ref{claim:one-triangle-neighbor},
        \(x,y\in N(v)\).
        If \(|N(u)\cap V(T_u)|=2\), then suppose, without loss of generality, that \(c\notin N(u)\).
        In this case, put \(S = \{ua,bc\}\).
        If \(|N(u)\cap V(T_u)|\neq 2\), by Claim~\ref{claim:no-unique-even-neighbor}, we have \(|N(u)\cap V(T_u)|=3\), then let \(S = \{ua,ub,uc\}\).
Let \(G_0 = G-\{uv,xv,yz\} - S\).
        Note that \(d_{G_0}(w) = d_G(w)-1\) for every \(w\in\{a,b,c,x,y,z\}\),
        \(d_{G_0}(v) = d_G(v)-2\), 
        and \(d_{G_0}(u)\in\{d_G(u)-2,d_G(u)-4\}\),
        and hence \(d_{G_0}(w)\) is odd 
for every \(w\in\{a,b,c,x,y,z,u,v\}\).
By Claim~\ref{claim:no-remaining-only-odds},
        \(\pn(G_0)\leq\floor{n / 2}\).

Let \(\D_0\) be a minimum path decomposition of \(G_0\).
In what follows, we restore the edges \(yz\), \(xv\), and \(uv\), in this order.
Note that every neighbor of \(y\) has odd degree in \(G_0\), thus, by Lemma~\ref{lemma:fan2}, \(yz\) is addible towards \(y\) with respect to \(\D_0\).
        Let \(\D_1\) be the \(yz\)-transformation of \(\D_0\) towards \(y\), and note that every neighbor of \(v\) in \(G_1=G_0+yz\) has odd degree, except for \(y\), but we have \(\D_1(y)\geq 2\).
Thus, by Lemma~\ref{lemma:fan2}, \(xv\) is addible towards \(v\) with respect to \(\D_1\).
        Let \(\D_2\) be the \(xv\)-transformation of \(\D_1\) towards \(v\), and note that every neighbor of \(u\) in \(G_2 = G_1+xv\) has odd degree.
Again, by Lemma~\ref{lemma:fan2}, \(uv\) is addible towards \(u\) with respect to \(\D_2\).
        Let \(\D_3\) be the \(uv\)-transformation of \(\D_2\) towards \(u\), and note that every neighbor of \(u\) in \(G_3 = G_2 + uv = G\setminus S\) has odd degree,
        and \(\D_3(u)\geq 2\).
        In what follows, we divide the proof on whether \(|N(u)\cap V(T_u)|=2\) or \(|N(u)\cap V(T_u)|\neq 2\).
        
        First, suppose that \(|N(u)\cap V(T_u)|=2\). 
        Note that \(a\) has no even neighbors in \(G_3\), and hence, by Lemma~\ref{lemma:fan2}, \(ua\) is addible towards \(a\) with respect to \(\D_3\).
        Let \(\D_4\) be the \(ua\)-transformation of \(\D_3\) towards \(a\), 
        and note that every neighbor of \(c\) in \(G_3\addeset ua\) is odd,
        except for \(a\), but \(\D_4(a) \geq 2\),
        and hence \(c\) has no passing neighbor in \(\D_4\).
        Thus, by Lemma~\ref{lemma:fan2}, \(bc\) is addible towards \(c\) with respect to \(\D_4\).
        But the \(bc\)-transformation \(\D\) of \(\D_4\) towards \(c\) is a path decomposition of \(G\) such that \(|\D|\leq\floor{n / 2}\), a contradiction (see Figure~\ref{fig:full-vertex3-1}).
            
        Thus, we may assume \(|N(u)\cap V(T_u)|\neq 2\).
Note that no neighbor of \(u\) in \(G_3=G\deleset S\) is a passing vertex in \(\D_3\).
        By Lemma~\ref{lemma:fan4}, there is \(B\subseteq S\) such that \(|B|\geq \lceil|S|/2\rceil=2\), and \(B\) is addible towards \(u\) with respect to \(\D_3\).
        Let \(\D_4\) be the \(B\)-transformation of \(\D_3\) towards \(u\), we have \(\D_4(u) \geq 1 + \lceil|S|/2\rceil\geq 3\).
        Note that \(S\setminus B\) contains at most one edge.
        If \(S\setminus B =\emptyset\), then \(\D_4\) is a path decomposition of \(G\) such that \(|\D_4|\leq\floor{n / 2}\), a contradiction.
        Thus, we may assume \(S\setminus B\neq \emptyset\).
        Suppose, without loss of generality, that \(S\setminus B =\{uc\}\) and let \(G_4 = G_3 + B\).
        Note that the only possible passing neighbors of \(c\) in \(\D_4\) are \(a\) and \(b\). 
        By Lemma~\ref{lemma:fan2}, \(uc\) is addible towards \(c\) with respect to \(\D_4\).
        Then, the \(uc\)-transformation \(\D\) of \(\D_4\) towards \(c\) is a decomposition of \(G\) such that \(|\D|\leq\floor{n / 2}\), a contradiction (see Figure~\ref{fig:full-vertex3-2}).
    \end{proof} 

    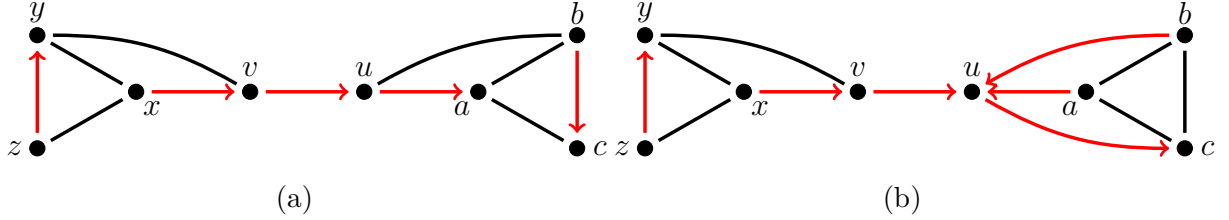
\begin{figure}[h]
        \centering
        \begin{subfigure}{.49\textwidth}
            \centering
            \begin{tikzpicture}[scale = 1.5]
	\node (v) [black vertex] 	at (0,0) {};
	\node (u) [black vertex]	at (1,0) {};
	
\node (x) [black vertex]	at (-1,0) {};
	\node (y) [black vertex]	at ($(x) + (150:1)$) {};
	\node (z) [black vertex] 	at ($(x) + (-150:1)$) {};
	
\node (a) [black vertex]	at ($(u)+(1,0)$) {};
	\node (b) [black vertex]	at ($(a) + (30:1)$) {};
	\node (c) [black vertex] 	at ($(a) + (-30:1)$) {};

\node () []	at ($(x) + (-45:0.2)$) {$x$};
	\node () []	at ($(y) + (90:0.2)$) {$y$};
	\node () [] 	at ($(z) + (180:0.2)$) {$z$};
	\node () []	at ($(a) + (180+45:0.2)$) {$a$};
	\node () []	at ($(b) + (90:0.2)$) {$b$};
	\node () [] 	at ($(c) + (0:0.2)$) {$c$};
	\node () []	at ($(v) + (90:0.2)$)  {$v$};
	\node () []	at ($(u) + (90:0.2)$)  {$u$};
	
	\draw[edge,color1,->] (v) -- (u);
	\draw[edge,color1,->] (x) -- (v);
	\draw[edge,color1,->] (z) -- (y);
	\draw[edge,color1,->] (b) -- (c);
	\draw[edge,color1,->] (u) -- (a);
	\draw[edge]	   (y) -- (x) -- (z) (c) -- (a) -- (b)  (u) to [bend left=15] (b);
	\draw[edge] (v) to [bend right=15] (y);
\end{tikzpicture}

             \caption{}
            \label{fig:full-vertex3-1}
        \end{subfigure}
        \begin{subfigure}{.49\textwidth}
            \centering
            \begin{tikzpicture}[scale = 1.5]
	\node (v) [black vertex] 	at (0,0) {};
	\node (u) [black vertex]	at (1,0) {};
	
\node (x) [black vertex]	at (-1,0) {};
	\node (y) [black vertex]	at ($(x) + (150:1)$) {};
	\node (z) [black vertex] 	at ($(x) + (-150:1)$) {};
	
\node (a) [black vertex]	at ($(u)+(1,0)$) {};
	\node (b) [black vertex]	at ($(a) + (30:1)$) {};
	\node (c) [black vertex] 	at ($(a) + (-30:1)$) {};

\node () []	at ($(x) + (-45:0.2)$) {$x$};
	\node () []	at ($(y) + (90:0.2)$) {$y$};
	\node () [] 	at ($(z) + (180:0.2)$) {$z$};
	\node () []	at ($(a) + (180+45:0.2)$) {$a$};
	\node () []	at ($(b) + (90:0.2)$) {$b$};
	\node () [] 	at ($(c) + (0:0.2)$) {$c$};
	\node () []	at ($(v) + (90:0.2)$)  {$v$};
	\node () []	at ($(u) + (90:0.2)$)  {$u$};
	
	\draw[edge,color1,->] (v) -- (u);
	\draw[edge,color1,->] (x) -- (v);
	\draw[edge,color1,->] (z) -- (y);
	\draw[edge] (c) -- (b);
	\draw[edge,color1,->] (a) -- (u);
	\draw[edge,color1,->] (b) to [bend right=15] (u);
	\draw[edge,color1,->] (u) to [bend right=15] (c);
	\draw[edge]	   (y) -- (x) -- (z) (c) -- (a) -- (b);
	\draw[edge] (v) to [bend right=15] (y);
\end{tikzpicture}

             \caption{}
            \label{fig:full-vertex3-2}
        \end{subfigure}
        \caption{Illustration of the proof of  Claim~\ref{claim:full-vertices3}.
                    The edges in \(\{uv,xv,yz\}\cup S\) are highlighted in red.
                    A directed edge \(\Vec{ab}\) indicates the use of Lemmas~\ref{lemma:fan2} and~\ref{lemma:fan4} adding the edge \(ab\) towards~\(b\).}
        \label{fig:full-vertex3}
    \end{figure}

    Recall that \(T_1,\ldots, T_s\) are the triangle components of \(\EV{G}\).
    Note that, given a graph \(G\), if \(s = 1\), then \(\EV{G}\) consists of one triangle.
    By Claim~\ref{claim:no-free-odd-vertex}, every odd vertex of \(G\) has an even neighbor,
    and by Claim~\ref{claim:no-unique-even-neighbor}, every odd vertex of \(G\)
    has at least two even neighbors.
    Thus, \(G\) is a SET graph, 
    a contradiction to the choice of \(G\).  
    Thus, we may assume \(s\geq 2\).

    Now, let \(P\) be a shortest path joining vertices of two different components of \(\EV{G}\).
By Claims~\ref{claim:only-triangles},~\ref{claim:one-triangle-neighbor}, and~\ref{claim:no-free-odd-vertex}, \(P\) contains precisely two internal vertices, say \(u\) and \(v\).
    Suppose, without loss of generality, that \(T_1\) is the triangle neighbor of \(u\),
    and \(T_2\) is the triangle neighbor of~\(v\).
Let \(V(T_1) = \{a,b,c\}\) and \(V(T_2) = \{x,y,z\}\).
    By Claim~\ref{claim:full-vertices3}, the vertices \(u\) and \(v\) are full vertices. Let \(S_u =\{ua,ub,uc\}\) and \(S_v=\{vx,vy,vz\}\), and let \(G_0 = G - uv - S_u - S_v\).
    Clearly, every vertex in \(\{a,b,c,x,y,z,u,v\}\) has odd degree in \(G_0\).
By Claim~\ref{claim:no-remaining-only-odds},
    \(\pn(G_0)\leq\floor{n / 2}\).
Let \(\D_0\) be a minimum path decomposition of \(G_0\).

    In what follows, we obtain a path decomposition \(\D_3\) of \(G_3 = G_0+uv+S_u= G-S_v\) such that \(\D_3(u),\D_3(v) \geq 1\), and then we extend it to a path decomposition of \(G\).
    First, we obtain a path decomposition \(\D_2\) of \(G_2 = G_0+S_u\) such that \(\D_2(u)\geq 2\).
    By Lemma~\ref{lemma:fan4}, there is a \(B_u\subseteq S_u\) such that \(|B_u|\geq \lceil |S_u|/2\rceil\geq 2\) and \(B_u\) is addible towards \(u\) with respect to \(\D_0\).
    Let \(\D_1\) be the \(B_u\)-transformation of \(\D_0\) towards \(u\).
    We have \(\D_1(u) \geq 1 + \lceil|S_u|/2\rceil\geq 3\).
    Note that \(S_u\setminus B_u\) contains at most one edge.
    If \(S_u\setminus B_u =\emptyset\), then put \(\D_2 = \D_1\) is the desired decomposition.
    If \(S_u\setminus B_u\neq \emptyset\), then suppose \(uc\in S_u\setminus B_u\) and put \(G_1 = G_0+B_u\).
    Note that the only possible passing neighbors of \(c\) are \(a\) and \(b\).
    By Lemma~\ref{lemma:fan2}, \(uc\) is addible towards \(c\) with respect to \(\D_1\).
    Then, the \(uc\)-transformation \(\D_2\) of \(\D_1\) towards \(c\) is the desired decomposition.
    Now, note that every neighbor of \(v\) in \(G_2 = G_1+uc = G_0+S_u\) has odd degree, and hence, by Lemma~\ref{lemma:fan2}, \(uv\) is addible towards \(v\) with respect to \(\D_2\).
    Then, the \(uv\)-transformation \(\D_3\) of \(\D_2\) towards \(v\) is a path decomposition of \(G_3 = G_0+uv+S_u= G-S_v\) such that \(\D_3(u)\geq 1\) and \(\D_3(v) \geq 2\).

    By Lemma~\ref{lemma:fan4}, there is a \(B_v\subseteq S_v\) such that \(|B_v|\geq \lceil |S_v|/2\rceil\geq 2\) and \(B_v\) is addible towards \(v\) with respect to \(\D_3\).
    Let \(\D_4\) be the \(B_v\)-transformation of \(\D_3\) towards \(v\).
    We have \(\D_4(v) \geq 1 + \lceil|S_v|/2\rceil\geq 3\).
    Note that \(S_v\setminus B_v\) contains at most one edge.
    If \(S_v\setminus B_v =\emptyset\), then \(\D = \D_4\) is a path decomposition of \(G\) such that \(|\D|\leq\floor{n / 2}\).
    If \(S_v\setminus B_v\neq \emptyset\), then suppose \(vz\in S_v\setminus B_v\).
    Note that the only passing neighbors of \(z\) in  \(G_4 = G_3+B_v\) are \(x\) and \(y\).
    By Lemma~\ref{lemma:fan2}, \(vz\) is addible towards \(z\) with respect to \(\D_4\).
    Then, the \(vz\)-transformation \(\D\) of \(\D_4\) towards \(z\) is a decomposition of \(G\) such that \(|\D|\leq\floor{n / 2}\).
\end{proof}

The next corollary then is a straightforward application of Lemma~\ref{lemma:ESET-special-decomposition} and Theorem~\ref{thm:main-theorem-1}.

\begin{corollary}
    If $G$ is a connected graph with \(\Delta(\EV{G}) \leq 3\), then \(\pn(G) \leq \ceil{|V(G)|/2}\).
\end{corollary}

\section{Towards Gallai's Conjecture}\label{sec:further}

In this section, we use Theorem~\ref{thm:main-theorem-1}
to prove Theorem~\ref{thm:theorems-strong},
which extends Fan's result~\cite{Fan05}.
Recall that \(\mathcal{G}\) denotes \calG.

\newcommand{\alternativecondition}[1]{$EV(#1)$ is a subgraph of a graph in \(\mathcal{G}\)}

\newcommand{\alternativeconditiongood}[1]{$EV(#1)$ is a subgraph of a graph in which 
									each block has maximum degree at most \(3\)
									and each component has at most one block that is not a triangle-free graph}

\newcommand{\alternativeconditionextraold}[2]{$EV(#1)$ is a subgraph of a graph $#2$ in which 
									each block has maximum degree at most \(3\)
									and each component has at most one block that is not a triangle-free graph}

\newcommand{\alternativeconditionextra}[2]{$EV(#1)$ is a subgraph of a graph $#2$ in which 
	each block has maximum degree at most~\(3\); 
    and each component either has maximum degree at most \(3\) or has at most one block that contains triangles}

\begin{theorem}\label{thm:further}
	If \(G\) is a connected simple graph on \(n\) vertices such that
\alternativecondition{G},
	then \(G\) is a Gallai graph or \(G\) is a SET graph.
\end{theorem}

\begin{proof}Suppose that the statement does not hold, 
	and let \(G\) be a graph 
	on \(n\) vertices such that
\alternativecondition{G}.
	Suppose that \(G\) is a counterexample that minimizes \(|E(G)|\).
	In what follows, we prove three claims regarding \(G\).	
	The proof of Claims~\ref{claim:further-no-hanging-ESET} and~\ref{claim:further-no-unique-even-neighbor}
	are analogous to the proofs of Claims~\ref{claim:no-hanging-ESET} and~\ref{claim:no-unique-even-neighbor}, respectively.
	We present their proof for completeness.
First, we prove that every hanging ESET subgraph of \(G\) must be connected at a special vertex.
	
	\begin{claim}\label{claim:further-no-hanging-ESET}
	    Let \(K\) be a hanging ESET subgraph of \(G\),
	    and let \(G'\) be such that \(G = K\cup G'\), and \(V(K)\cap V(G')=\{u\}\).
Then \(u\) is an odd vertex of \(G\).
	\end{claim}
	
	\begin{proof}
	    Let \(K\), \(G'\), and \(u\) be as in the statement.
Suppose, for a contradiction, that \(u\) has even degree in \(G\).
	    Then \(\EV{G'}\subseteq\EV{G}\),
	    and hence 
\alternativecondition{G'}.
	    By the minimality of \(G\), the graph \(G'\) is either a Gallai graph or a SET graph.
	    First, suppose that \(G'\) is a Gallai graph, 
	    i.e., \(\pn(G')\leq\lfloor|V(G')|/2\rfloor\).
	    By Lemma~\ref{lemma:absorbing-all}, it follows that
	    \begin{align*}
	        \pn(G) &\leq \big\lceil|V(K)|/2\big\rceil + \big\lfloor|V(G')|/2\big\rfloor-1\\
	                &\leq \big(|V(K)|+1\big)/2 + |V(G')|/2-1\\
	                &=\big(|V(K)|+|V(G')|-1\big)/2\\
	                &=|V(G)|/2.
	    \end{align*}
        Therefore,  \(\pn(G)\leq \big\lfloor |V(G)|/2\big\rfloor\),
        and \(G\) is a Gallai graph,
        a contradiction.
	    
	    Thus, we may assume that \(G'\) is a SET graph.
	    By Lemma~\ref{lemma:ESET-special-decomposition},
	    \(K\) (resp.  \(G'\)) admits a path decomposition \(\D_K\) (resp.\ \(\D'\)) 
	    such that \(\D_K(u)\geq 2\) and \(|\D_K|\leq\big\lceil|V(K)|/2\big\rceil\) (resp.\ \(\D'(u)\geq 2\)
	    and \(|\D'|\leq \big\lceil |V(G')|/2\big\rceil\)).
	    Let \(P_1\) and \(P_2\) be paths in \(\D_K\) having \(u\) as end vertex,
	    and \(Q_1\) and \(Q_2\) be paths in \(\D'\) having \(u\) as end vertex.
	    Put \(R_1 = P_1\cup Q_1\) and \(R_2=P_2\cup Q_2\),
	    and note that \(\D = \big(\D_K\setminus\{P_1,P_2\}\big)\cup\big(\D'\setminus\{Q_1,Q_2\}\big)\cup\{R_1,R_2\}\)
	    is a path decomposition of \(G\) with cardinality
	    \begin{align*}
	        |\D| &\leq \big\lceil |V(K)|/2\big\rceil + \big\lceil|V(G')|/2\big\rceil -2 \\
	            &\leq \big(|V(K)|+1\big)/2+\big(|V(G')|+1\big)/2-2\\
	            &\leq \big(|V(K)|+|V(G')|-2\big)/2 \\
	            &< \big(|V(K)|+|V(G')|-1\big)/2 \\
                &=|V(G)|/2.
	    \end{align*}
        Therefore,  \(\pn(G)\leq \big\lfloor |V(G)|/2\big\rfloor\),
        and \(G\) is a Gallai graph,
        a contradiction.
	\end{proof}

    Now use Fan's techniques to prove that \(\EV{G}\) consists of vertex-disjoint triangles.
    First, we prove that no vertex of \(G\) has a unique even neighbor.

    \begin{claim}\label{claim:further-no-unique-even-neighbor}
        No vertex of \(G\) has exactly one even neighbor.
    \end{claim}

    \begin{proof}
        Suppose, for a contradiction, that \(G\) contains a vertex \(u\) that has precisely one even neighbor, say \(v\), and let \(G'=G \setminus uv\).
        Note that \(v\) has odd degree in \(G'\) and \(u\) has no even neighbor in \(G'\).
        Therefore, \(\EV{G'}-u\subseteq \EV{G}\) and if \(d_G(u)\) is odd, 
		then \(u\) is an isolated vertex in \(\EV{G'}\).
		Thus, 
\alternativecondition{G'}.
We claim that \(G'\) is a Gallai graph.
        By Remark~\ref{remark:gallai-components}, it is enough to prove that no component of \(G'\)
        is a SET graph.
Indeed, \(G'\) has at most two components, say \(G'_u\) and \(G'_v\), 
        that contain, respectively, \(u\) and \(v\).
Since \(u\) has no even neighbors in \(G'\), \(G'_u\) is not a SET graph.
Thus, if \(G'\) is connected, i.e., \(G'_u=G'_v\),
        then \(G'\) is a Gallai graph, as desired.
        Thus, we may assume \(G'_u\neq G'_v\).
        In this case, note that, if \(G'_v\) is a SET graph,
        then \(G'_v\) is a hanging ESET subgraph of \(G\) connected at \(v\),
        which is an even vertex of \(G\), a contradiction to Claim~\ref{claim:further-no-hanging-ESET}.
        Thus, \(G'_u\) and \(G'_v\) are Gallai graphs as desired.

        Let \(\sD'\) be a minimum path decomposition of \(G'\).
        Since \(v\) has odd degree in \(G'\), it follows \(\D'(v) \geq 1\), and since \(u\) has no even neighbor in \(G'\), we have \(|\{x\in N_{G'}(u)\colon \D'(x)=0\}|=0\).
        Thus, by Lemma~\ref{lemma:fan2}, \(uv\) is addible towards \(u\) with respect to \(\D'\), a contradiction.
    \end{proof}
    
    The next claim follows the steps introduced by Fan~\cite{Fan05}.

    \begin{claim}\label{claim:further-only-triangles}
        Every leaf block of \(\EV{G}\) is a triangle or an isolated vertex.
    \end{claim}

    \begin{proof}
Suppose, for a contradiction, that \(\EV{G}\) contains a component \(C\) which is neither a triangle nor an isolated vertex, and let \(H\) be a leaf block of \(C\).
If \(H \neq C\), then let \(w\) be the cut vertex of \(C\) in \(H\),
        otherwise, let \(w\) be any vertex of \(H\).

        \begin{subclaim}
          \(H\) is a cycle.
        \end{subclaim}

        \begin{proof}
            By Claim~\ref{claim:further-no-unique-even-neighbor}, no vertex of \(H\) has E-degree \(1\).
We prove that \(d_H(u)\leq 2\) for every \(u\in V(H)\setminus\{w\}\).
            Suppose, for a contradiction, that \(H\) has a vertex \(u\neq w\) with degree~\(3\), and let \(\{v_1, v_2, v_3\} \subseteq N_H(u)\) be three even neighbors of \(u\).
            Let \(F\) be the subgraph of \(G\) induced by the edges \(uv_1, uv_2, uv_3\),
            and let \(G' = G \setminus E(F)\).
            Note that every vertex in \(V(F)\) has odd degree in \(G'\) and, hence, that \(\EV{G'} \subset \EV{G}\).
            Thus, it follows that 
\alternativecondition{G'}.
We claim that no component of \(G'\) is a SET graph.
			Indeed, suppose that some component \(C'\) of \(G'\) is a SET graph.
			Let \(x\), \(y\), \(z\) be even vertices of \(C'\).
Suppose that \(C'\) contains at least two vertices in \(\{u,v_1,v_2,v_3\}\), say \(a\) and \(b\).
			Note that \(x\), \(y\), and \(z\) are vertices of \(H\).
			Since \(C'\) is a SET graph, \(a\) and \(b\) have a common neighbor in \(x,y,z\), say \(x\), but this implies that \(x\) has degree at least \(4\) in \(H\), a contradiction.
			Thus, \(C'\) contains at most one vertex in \(\{u,v_1,v_2,v_3\}\), say \(a\).
			This implies that \(C'\) is a hanging ESET of \(G\),
			but \(a\) has even degree in \(G\), a contradiction to Claim~\ref{claim:further-no-hanging-ESET}.		
			By Remark~\ref{remark:gallai-components}, \(G'\) is a Gallai graph.	
Let \(\D_1\) be a minimum path decomposition of \(G'\).
            Since the vertices of \(F\) have odd degree in \(G'\), it follows that \(\sD_1(v) \geq 1\) for every \(v \in V(F)\).
            By Lemma~\ref{lemma:fan4}, there is a set \(B\subseteq\{uv_1,uv_2,uv_3\}\) addible towards \(u\) with respect to \(\D_1\) and containing \textcolor{black}{at least} two edges.
            Moreover, if \(w\) is a neighbor of \(u\), say \(w=v_1\), then Lemma~\ref{lemma:fan4} guarantees that we can choose \(B\) such that \(uv_1\in B\).
            Let \(\D_2\) be the \(B\)-transformation of \(\D_1\) towards \(u\).
            If \(|B|=3\), then \(\D_2\) is a path decomposition of \(G\) such that \(|\D_2| = |\D_1| \leq\lfloor n/2\rfloor\), a contradiction.
            Thus, we may assume \(|B|=2\).
            Suppose, without loss of generality that \(uv_3\notin B\).
Note that \(\D_2(u)\geq 3\), and that, since \(d_{\EV{G}}(v_3) \leq 3\), we have \(d_{\EV{G'+B}}(v_3)\leq 2\), because \(u\) is not a neighbor of \(v_3\) in \(G'\).
            Thus, by Lemma~\ref{lemma:fan2}, \(uv_3\) is addible towards \(v_3\) with respect to \(\D_2\), a contradiction.
            Thus, every vertex in \(H\) different from \(w\) has degree~\(2\).
        \end{proof}
        
        Suppose that \(H\) is not a triangle, and let \(u\) be a neighbor of \(w\) in \(H\).
        Let \(v_1\) and \(v_2\) be the even neighbors of \(u\), where \(w=v_1\).
        Since \(H\) is not a triangle, \(v_1\) and \(v_2\) are not adjacent.
        Let \(w'\) be an even neighbor of \(v_2\) different from \(u\).  
Let \(G' = G \setminus \{uv_1,v_2w'\}\), and
        note that every vertex in \(\{u,v_1,v_2,w'\}\) has odd degree in \(G'\), 
        and hence \(\EV{G'}\subseteq \EV{G}\).
        Thus, 
\alternativecondition{G'}.
We claim that no component of \(G'\) is a SET graph.
		Indeed, suppose that some component \(C'\) of \(G'\) is a SET graph.
		Let \(x\), \(y\), \(z\) be even vertices of \(C'\).
Suppose that \(C'\) contains at least two vertices in \(\{u,v_1,v_2,v_3\}\), say \(a\) and \(b\).
		Note that \(x\), \(y\), and \(z\) are vertices of \(H\).
		Since \(C'\) is a SET graph, \(a\) and \(b\) have a common neighbor in \(x,y,z\), say \(x\), but this implies that \(x\) has degree at least \(4\) in \(B'\), a contradiction.
		Thus, \(C'\) contains at most one vertex in \(\{u,v_1,v_2,w'\}\), say \(a\).
		This implies that \(C'\) is a hanging ESET of \(G\),
		but \(a\) has even degree in \(G\), a contradiction to Claim~\ref{claim:further-no-hanging-ESET}.		
		By Remark~\ref{remark:gallai-components}, \(G'\) is a Gallai graph.	
Let \(\D_1\) be a minimum path decomposition of \(G'\).
        Note that \(v_2\) has no even neighbors in \(G'\).
        Thus, by Lemma~\ref{lemma:fan2}, \(v_2w'\) is addible towards \(v_2\) with respect to \(\D_1\).
        Let \(\D_2\) be the \(v_2w'\)-transformation of \(\D_1\)  towards \(v_2\).
        Note that \(v_2\) is the only even neighbor of \(u\) in \(G'+v_2w'\),
        but \(\D_2(v_2)\geq 2\).
        Thus, by Lemma~\ref{lemma:fan2}, \(uv_1\) is addible towards \(u\) with respect to \(\D_2\).
        Therefore, the \(uv_1\)-transformation \(\D\) of \(\D_2\)  towards \(u\) is a path decomposition of \(G\) 
        such that \(|\D|\leq\floor{n/2}\), a contradiction.
    \end{proof}

	Now, if every component of \(\EV{G}\) has maximum degree at most \(3\),
	the statement follows by Theorem~\ref{thm:main-theorem-1}.
	Thus, we may assume that \(\EV{G}\) contains a vertex \(u\) of degree at least \(4\).
	By hypothesis,  \alternativeconditionextra{G}{H}.
	Let \(H_u\) be the component of \(H\) that contains \(u\).
Since \(d_H(u)\geq 4\), \(u\) must be a cut vertex of \(H\) (and hence, a cut vertex of \(H_u\)).
	Let \(H'_1,\ldots,H'_k\) be the components of \(H_u\delv u\),
	and put \(H_i = H[V(H'_i)\cup\{u\}]\) for \(i=1,\ldots,k\).
	Clearly, \(H_1,\ldots, H_k\) decompose \(H_u\).
	Since \(d_{\EV{G}}(u)\geq 4\),
	at least two of these graphs, say \(H_i\) and \(H_j\), contain edge of \(\EV{G}\) incident to \(u\).
Let \(H''_i\) and \(H''_j\) be the subgraphs of \(\EV{G}\) contained in \(H_i\) and \(H_j\), respectively.
	By Claim~\ref{claim:further-only-triangles}, every leaf block of \(\EV{G}\) in \(H''_i\) and in \(H''_j\) is a triangle.
	This implies that each \(H_i\) and \(H_j\) contain a block that is not a triangle-free graph,
	and hence \(H_u\) contains at least two blocks which are not triangle-free graphs,
	a contradiction.
\end{proof}

\section{Concluding remarks}\label{sec:concluding}

In this paper we give a step towards verifying Conjecture~\ref{conj:gallai}.
In fact, Theorems~\ref{thm:main-theorem-1} and~\ref{thm:further} present statements which are in between Conjecture~\ref{conj:gallai}
and Conjecture~\ref{conj:strong-gallai}.
This indicates that an intermediate statement may be easier to deal than Conjectures~\ref{conj:gallai} and~\ref{conj:strong-gallai}.
In order to strengthen the results in this paper for fitting Conjecture~\ref{conj:strong-gallai},
one need only to verify Conjecture~\ref{conj:strong-gallai} for SET graphs.
In fact, knowing that the only non-Gallai graphs with E-degree at most \(3\)
are the odd semi-cliques would simplify the proof of Claim~\ref{claim:no-fan-subgraph},
which was introduced to deal with SET subgraphs.
By using the Integer Linear Formulation presented in~\cite{BoCaSa-Lagos19},
we were able to check this fact for SET graphs up to eleven vertices.

This work benefited greatly from the techniques introduced by Fan~\cite{Fan05},
and 
there are two directions that we believe to be natural for extending the results presented in this paper,
i.e., two graph classes for which Conjectures~\ref{conj:gallai} and~\ref{conj:strong-gallai} are worth exploring 
with the techniques introduced here and in~\cite{Fan05}.

\begin{enumerate}[1.]\setlength\itemsep{0em}
	\item	Graphs with E-degree at most \(4\);
	\item	Graphs in which each block of \(\EV{G}\) has maximum degree at most \(3\).
\end{enumerate}

\bibliographystyle{amsplain}

\end{document}